\documentclass[letterpaper, 11pt,  reqno]{amsart}

\usepackage{amsmath,amssymb,amscd,amsthm,amsxtra, esint}

\usepackage[implicit=true]{hyperref}

\usepackage[left=32mm, right=32mm, 
bottom=27mm]{geometry}
\allowdisplaybreaks[2]

\usepackage{color}

\definecolor{gr}{rgb}   {0.,   0.69,   0.23 }
\definecolor{bl}{rgb}   {0.,   0.5,   1. }
\definecolor{mg}{rgb}   {0.85,  0.,    0.85}
\definecolor{yl}{rgb}   {0.8,  0.7,   0.}
\definecolor{or}{rgb}  {0.7,0.2,0.2}

\newtheorem{theorem}{Theorem} [section]

\newtheorem{lemma}[theorem]{Lemma}
\newtheorem{proposition}[theorem]{Proposition}
\newtheorem{remark}[theorem]{Remark}

\newtheorem{corollary}[theorem]{Corollary}

\newtheorem*{ack}{Acknowledgments}

\newtheorem{conjecture}{Conjecture}

\DeclareMathOperator*{\intt}{\int}

\DeclareMathOperator{\MAX}{MAX}

\newcommand{\1}{\hspace{0.5mm}\text{I}\hspace{0.5mm}}
\newcommand{\II}{\text{I \hspace{-2.8mm} I} }
\newcommand{\III}{\text{I \hspace{-2.9mm} I \hspace{-2.9mm} I}}

\newcommand{\noi}{\noindent}
\newcommand{\Z}{\mathbb{Z}}
\newcommand{\R}{\mathbb{R}}

\newcommand{\T}{\mathbb{T}}

\let\Re=\undefined\DeclareMathOperator*{\Re}{Re}
\let\Im=\undefined\DeclareMathOperator*{\Im}{Im}

\let\P= \undefined
\newcommand{\P}{\mathbf{P}}
\newcommand{\PP}{\mathbb{P}}

\newcommand{\E}{\mathbb{E}}
\newcommand{\D}{\mathcal{D}}

\renewcommand{\L}{\mathcal{L}}
\newcommand{\M}{\mathcal{M}}

\newcommand{\N}{\mathcal{N}}
\newcommand{\NB}{\mathbb{N}}

\newcommand{\F}{\mathcal{F}}

\newcommand{\al}{\alpha}
\newcommand{\be}{\beta}

\newcommand{\dl}{\delta}

\newcommand{\eps}{\varepsilon}

\newcommand{\g}{\gamma}

\newcommand{\ld}{\lambda}

\newcommand{\s}{\sigma}
\newcommand{\Si}{\Sigma}
\newcommand{\ft}{\widehat}

\newcommand{\wt}{\widetilde}
\newcommand{\cj}{\overline}
\newcommand{\dx}{\partial_x}
\newcommand{\dt}{\partial_t}
\newcommand{\dd}{\partial}

\newcommand{\ta}{\theta}

\renewcommand{\l}{\ell}
\renewcommand{\o}{\omega}
\renewcommand{\O}{\Omega}

\newcommand{\les}{\lesssim}
\newcommand{\ges}{\gtrsim}

\newcommand{\jb}[1]
{\langle #1 \rangle}

\newcommand{\ind}{\mathbf{1}}

\newcommand{\low}{\textup{low}}
\newcommand{\high}{\textup{high}}
\newcommand{\HS}{\textup{HS}}

\DeclareMathOperator{\Id}{Id}

\DeclareMathOperator{\Law}{Law}

\numberwithin{equation}{section}
\numberwithin{theorem}{section}

\newcommand{\too}{\longrightarrow}

\makeatletter
\@namedef{subjclassname@2020}{%
  \textup{2020} Mathematics Subject Classification}
\makeatother

\begin{document}
\baselineskip = 14pt

\title[Global dynamics for SKdV]{Global dynamics for the stochastic KdV equation with white
noise as initial data}

\author[T.~Oh, J.~Quastel, and P.~Sosoe]
{Tadahiro Oh, Jeremy Quastel, and 
Philippe Sosoe}

\address{
Tadahiro Oh\\
 School of Mathematics\\
The University of Edinburgh\\
and The Maxwell Institute for the Mathematical Sciences\\
James Clerk Maxwell Building\\
The King's Buildings\\
Peter Guthrie Tait Road\\
Edinburgh\\ 
EH9 3FD\\
 United Kingdom}

\email{hiro.oh@ed.ac.uk}

\address{
Jeremy Quastel\\
Departments of Mathematics and Statistics\\
University of Toronto\\
40 St. George St\\
Toronto, ON M5S 2E4, Canada,
and
School of Mathematics\\
Institute for Advanced Study\\
Einstein Drive, Princeton\\ NJ 08540\\ USA}
\email{quastel@math.toronto.edu}

\address{Philippe Sosoe\\
Department of Mathematics\\
Cornell University\\ 
310 Malott Hall\\ 
Cornell University\\
 Ithaca\\ New York 14853\\ USA}

\email{psosoe@math.cornell.edu}



\begin{abstract}
We study 
the stochastic 
 Korteweg-de Vries equation (SKdV)  with an additive space-time white noise forcing, 
 posed on the one-dimensional torus.
In particular, we construct global-in-time solutions to SKdV 
with spatial white noise initial data. 
Due to the lack of an invariant measure, 
 Bourgain's invariant measure argument is not applicable to this problem.
In order to overcome this difficulty, we implement a variant of  Bourgain's argument in the context of 
an evolution system of measures and construct global-in-time dynamics.
Moreover, we show that 
the white noise measure with variance $1+t$ is an evolution system of measures
for SKdV with the white noise initial data.

\end{abstract}


\maketitle

\tableofcontents

\newpage

\section{Introduction}\label{SEC:1}

\subsection{Main result}
The main objective of the present paper is to explain how techniques developed to study 
 invariance of certain measures (in our case, a spatial white noise) under the flow of 
 Hamiltonian partial differential equations (PDEs) can be combined with the analysis of stochastic perturbations of these equations to construct global-in-time solutions in a probabilistic setting.

In particular, we consider the following Cauchy problem for the stochastic Korteweg-de Vries equation (SKdV)
on the one-dimensional torus $\T = \R/(2\pi \Z)$:
\begin{align}
\begin{cases}
\dt u + \dx^3 u + u \dx u = \xi\\
u|_{t = 0} = u_0.
\end{cases}
\label{KdV1}
\end{align}

\noi
Here, $\xi$ denotes an additive (Gaussian) space-time white noise forcing
whose space-time covariance is (formally) given by 
\begin{align}
 \E[ \xi(x_1, t_1)\xi(x_2, t_2) ] = \dl(x_1 - x_2) \dl (t_1 - t_2)
\label{white1}
\end{align} 

\noi
for $x_1, x_2 \in \T$ and $t_1, t_2 \in \R_+$ 
with $\dl$ denoting the Dirac delta function.
In particular, we study~\eqref{KdV1}
with  a  spatial  white noise\footnote{As it is customary in the literature, 
with a slight abuse of notation, 
we use the term `white noise' to refer to both the distribution-valued random variable $u_0^\o$ in \eqref{series1}
and its law $\mu_1 = \Law (u_0^\o)$, when there is no confusion.
Here, $\Law(X)$ denotes the law of a random variable $X$.
For clarity, we may refer to $\mu_1 = \Law (u_0^\o)$ as the white noise measure.} on $\T$, independent of the forcing  $\xi$,  as initial data.
More concretely, we take $u_0 = u_0^\o$ of the form:\footnote{By convention, 
 we endow $\T$ with the normalized Lebesgue measure $(2\pi)^{-1} dx$.}
\begin{align}
u_0^\o(x) = \sum_{n \in \Z} g_n(\o) e^{inx}, 
\label{series1}
\end{align}

\noi
where $\{g_n \}_{n \in \Z}$ is a family of independent standard 
complex-valued Gaussian random variables conditioned that $g_{-n} = \cj{g_n}$, $n \in \Z$.
The main difficulty of this problem comes
from the roughness of the noise
and the white noise initial data, 
such that the solution $u(t)$ to~\eqref{KdV1}
belong to $H^s(\T) \setminus H^{-\frac 12} (\T)$, $s < - \frac 12$, almost surely.
Here, 
 $H^s(\mathbb{T})$ denotes the $L^2$-based Sobolev space
  defined  by the norm:
\[\|u\|_{H^s}=\bigg(\sum_{n\in\mathbb{Z}} \jb{n}^{2s}|\ft u(n)|^2\bigg)^\frac 12,\]

\noi
where $\jb{\,\cdot\,} = \sqrt {1 + |\cdot|^2}$.

The well-posedness issue of SKdV with an additive forcing:
\begin{align}
\dt u + \dx^3 u + u \dx u = \phi \xi, 
\label{KdV2}
\end{align}

\noi
where $\phi$ is a bounded operator on $L^2$, 
 has 
been studied  both on the real line and on the torus
\cite{DD1, ddt-2, Prin,  ddt-1, OH4}.
In the periodic setting,  de Bouard, Debussche, and Tsutsumi~\cite{ddt-1}
proved local well-posedness of \eqref{KdV2} on $\T$
when $\phi$ is a Hilbert-Schmidt operator from $L^2(\T)$ to $H^s(\T)$
for $s > -\frac{1}{2}$, barely missing the case of an additive space-time white noise.
This local well-posedness result  in~\cite{ddt-1} was obtained via  a contraction argument, 
based on the  Fourier restriction norm method (namely, utilizing the $X^{s, b}$-spaces) adapted to the Besov space, 
utilizing the endpoint Besov regularity of the Brownian motion
\cite{C1, Roy, Benyi1}.
With an additional assumption that $\phi$ is Hilbert-Schmidt from $L^2(\T)$
to $L^2(\T)$, they also proved global well-posedness of \eqref{KdV2} in $L^2(\T)$.
In~\cite{OH4}, the first author
improved this result and proved local well-posedness of~\eqref{KdV2}
even when $\phi = \Id$ (thus reducing to \eqref{KdV1}), 
thus handling the case of an additive  space-time white noise.\footnote{Note that 
$\phi = \Id$ is a Hilbert-Schmidt operator from $L^2(\T)$ to $H^s(\T)$
for $s < - \frac 12$ but not for $s \geq - \frac 12$.}
We point out that
 the argument in \cite{OH4} is based
on an approximation argument, in particular, not based
on a contraction argument.
Below, we will describe the approach in \cite{OH4} more in detail;
see Section \ref{SEC:LWP}.
Our main goal is to construct 
global-in-time dynamics for \eqref{KdV1}
with the spatial white noise $u_0^\o$ in~\eqref{series1} as initial data.

Before proceeding further, 
let us  go over the known well-posedness results for  the (deterministic) KdV on $\T$:
\begin{equation} 
\dt u + \dx^3 u + u \dx u = 0.
\label{KDV}
\end{equation}

\noi
In \cite{BO1}, Bourgain  introduced 
the so-called Fourier restriction norm method, 
utilizing  the $X^{s, b}$-spaces defined by the norm:
\begin{equation} \label{Xsb}
\| u \|_{X^{s, b}(\mathbb{T} \times \mathbb{R})} = \| \jb{n}^s \jb{\tau - n^3}^b 
\ft{u}(n, \tau) \|_{\l^2_n L^2_{\tau}(\mathbb{Z} \times \R)},
\end{equation}

\noindent
and proved  local well-posedness of \eqref{KDV} in $L^2(\mathbb{T})$
via a fixed point argument, 
immediately yielding global well-posedness in $L^2(\mathbb{T})$
thanks to the conservation of the $L^2$-norm.
Subsequently, 
Kenig, Ponce, and Vega \cite{KPV4} (also see \cite{CKSTT4})
improved Bourgain's result 
and proved  local well-posedness of \eqref{KDV} in $H^{-\frac{1}{2}}(\T)$
by establishing the following bilinear estimate:
\begin{equation}
\| \dx(uv) \|_{X^{s, b -1}} \lesssim \| u \|_{X^{s, b}} \| v \|_{X^{s, b}}
 \label{bilin1}
\end{equation}

\noi
for $s \geq -\frac{1}{2}$ and $b = \frac{1}{2}$
under the (spatial) {\it mean-zero} assumption on $u$ and $v$.
In~\cite{CKSTT4}, Colliander, Keel, Staffilani, Takaoka, and Tao then  proved 
the corresponding global well-posedness result in $H^{-\frac{1}{2}}(\T)$ via the $I$-method. 
The KdV equation \eqref{KDV} is also known
to be  one of the simplest completely integrable PDEs,
and there are well-posedness results for \eqref{KDV}, 
exploiting the completely integrable structure  of the equation.
In \cite{BO3}, Bourgain proved  global well-posedness of \eqref{KDV}
in the class $\M(\T)$ of finite Borel measures  $\ld $ on $\T$, assuming that its total variation $\|\ld\|$ is sufficiently small.
His proof was based on partially iterating the Duhamel formulation of \eqref{KDV}
and establishing bilinear and trilinear estimates, assuming an a priori uniform bound 
of the form: 
\begin{equation} \label{BOO}
\sup_{t \in \R}
\sup_{n\in \mathbb{Z}} |\ft{u}(n, t)| \le C
\end{equation}

\noi
on the Fourier coefficients of the solution $u$.
Then, he established the global-in-time a priori bound~\eqref{BOO}, using the complete integrability.
In \cite{KT}, Kappeler and Topalov  proved  global well-posedness of \eqref{KDV} in $H^{-1}(\T)$
via the inverse spectral method.
See also \cite{KV}.

For SKdV  \eqref{KdV1} with a random perturbation,  
such an  integrable structure is destroyed 
and thus the approaches based on the complete integrability of KdV are no longer applicable.
Nonetheless, in \cite{OH4}, the first author adapted Bourgain's approach \cite{BO3}, 
based on a partial iteration of the Duhamel formulation (= the mild formulation)
of \eqref{KdV1}, 
and proved local well-posedness of \eqref{KdV1}.
In particular, he bypassed the assumption~\eqref{BOO}
by employing the Fourier restriction norm method adapted
to the ``Fourier-Besov''  space~$\ft b^s_{p, \infty}(\T)$ 
introduced in \cite{OH1}, defined by the norm:
\begin{align}
\begin{split}
\|f\|_{\widehat{b}^s_{p,\infty}}=\|\ft f\|_{b^s_{p,\infty}}
& = \sup_{j \in \Z_{\ge 0} }\| \jb{n}^s \ft f(n)\|_{\l^p_{|n|\sim 2^j}}\\
& =\sup_{j \in \Z_{\ge 0} }\bigg(\sum_{|n|\sim 2^j}\langle n \rangle^{sp}|\ft f(n)|^p\bigg)^{\frac 1p}, 
\end{split}
\label{bs}
\end{align}

\noi
which captures the spatial regularity of the space-time white noise
when $sp < -1$; see Proposition 3.4 in \cite{OH1}.\footnote{In other words, 
$\phi = \Id$ is a $\g$-radonifying operator from $L^2(\T)$ to $\ft b^s_{p, \infty}(\T)$
when $sp < -1$, 
which is a suitable generalization of  the notion of  Hilbert-Schmidt operators in the Banach space setting;
see  \cite{BP, vNW}. See also \cite[Chapter 9]{HNVW}.}
Here, $\Z_{\ge 0} = \NB \cup \{0\}$, 
and $\{|n|\sim 2^j\}$ means
$\{2^{j-1} < |n| \le 2^j\}$ when $j \ge 1$
and 
$\{|n| \leq 1\}$ when $j = 0$.
Note that, by taking $p > 2$ (but close to $2$), 
we can take $s > -\frac 12$, still satisfying $sp < -1$, 
which is crucial in establishing relevant nonlinear estimates.
In Section \ref{SEC:LWP}, 
we go over some aspects of the local well-posedness argument from~\cite{OH4}.

\medskip

We now state our main result, 
which extends 
the solution constructed in \cite{OH4} 
globally in time in the case of the  white noise initial data.
We say that $u$ is a solution to \eqref{KdV1} if it satisfies the following Duhamel formulation (= the
mild formulation):
\begin{equation}
u(t)=S(t)u_0-\frac{1}{2}\int_0^t S(t-t')\partial_x u^2(t')d t+\int_0^t S(t-t')d W(t'), 
\label{KdV3}
\end{equation}

\noi
where $S(t) = e^{-t\dx^3}$ denotes the linear KdV propagator (= the Airy propagator)
and $W$ denotes a cylindrical Wiener process on $L^2(\T)$:
\begin{align}
W(t)
 = \sum_{n \in \Z} \be_n (t) e^{inx}, 
\label{W1}
\end{align}

\noi
where 
$\{\be_n \}_{n \in \Z}$ 
is defined by 
$\be_n(t) = \jb{\xi, \ind_{[0, t]} \cdot e_n}_{ x, t}$.
Here, $\jb{\cdot, \cdot}_{x, t}$ denotes 
the duality pairing on $\T\times \R_+$.
As a result, 
we see that $\{ \be_n \}_{n \in \Z}$ is a family of mutually independent complex-valued
Brownian motions conditioned that $\be_{-n} = \cj{\be_n}$, $n \in \Z$. 
In particular, $\be_0$ is  a standard real-valued Brownian motion,
and 
   we have
 \begin{align}
 \text{Var}(\be_n(t)) = \E\big[
 \jb{\xi, \ind_{[0, t]} \cdot e_n}_{x, t}\cj{\jb{\xi, \ind_{[0, t]} \cdot e_n}_{x, t}}
 \big] = \|\ind_{[0, t]} \cdot e_n\|_{L^2_{x, t}}^2 = t
\label{W1a}
 \end{align}

\noi
for any $n \in \Z$.
Note that the space-time white noise $\xi$ in \eqref{KdV1} is a distributional time derivative of
the cylindrical Wiener process $W$ in \eqref{W1}.
The third term on the right-hand side of \eqref{KdV3} is 
the so-called stochastic convolution, representing the 
effect of the stochastic forcing.

In the following, we set
\begin{align}
s = -\frac 12 + \dl_1 \qquad \text{and} \qquad p = 2 + \dl_2
\label{cond1}
\end{align}

\noi
for some small $\dl_1, \dl_2 > 0$
such that $sp < -1$.
Given $\al \ge 0$,\footnote{By convention, we have $X\equiv 0$ when $\al = 0$.
Namely, $\mu_0 = \dl_0$, where $\dl_0$ is the Dirac delta distribution at the trivial function.} we say that a distribution-valued random variable $X$ on $\T$ (and its law, denoted by $\mu_\al$)
is a (spatial) white noise on~$\T$ with variance~$\al$ if
\begin{align}
 \mu_{\al} = \Law(X) = \Law (\sqrt \al\, u_0^\o),
 \label{mu1}
\end{align}

\noi
where $u_0^\o$ is the white noise (with variance 1) in \eqref{series1}.

\begin{theorem}\label{THM:1}
The stochastic KdV equation \eqref{KdV1} with an additive space-time white noise
forcing is globally well-posed with white noise initial data.
More precisely,  
 there exist small $\dl_1, \dl_2 > 0$ 
 such that,  with probability $1$,
 there exists
  a unique global-in-time solution
 $u$ to~\eqref{KdV1},
 belonging to the class  
  $C(\R_+; \ft b^s_{p,\infty} (\T))$
  with $s$ and $p$ as in \eqref{cond1}, 
 with the white noise initial data  $u_0^\o$ in \eqref{series1}.
   Moreover, for any $t\ge 0$, 
 we have
\begin{align}
\Law (u(t)) = \mu_{1+t}.
\label{th1}
\end{align}

\noi
Namely,   
 $u(t)$ is    a  white noise with variance $1+t$. 
\end{theorem}

The proof of Theorem \ref{THM:1}
is based on a variant of Bourgain's invariant measure argument~\cite{BO2}
in the context of an {\it evolution system of measures} \cite{DR, DD}, 
which is
a natural generalization of the concept of invariant measures for an autonomous dynamical system.
Let us  give a somewhat formal definition
 of  an evolution system of measures.
Let $\Phi_{t_1, t_2}= \Phi_{t_1, t_2}^\o$, $t_2 \ge t_1 \ge 0$,  be a solution map for a given autonomous (random) dynamical system, 
sending the data $\varphi$ at time $t_1$ to the solution $\Phi_{t_1, t_2}\varphi$ at time $t_2$.
Then, we define the transition semigroup $P_{t_1, t_2}$
by 
\begin{align}
 P_{t_1, t_2}F(\varphi) = \E[ F(\Phi_{t_1, t_2}^\o \varphi)]
\label{trans1}
\end{align} 

\noi
for a bounded measurable function $F$ on the phase space $\M$.
Then, we say that\footnote{Strictly speaking, an evolution system of measures
is the mapping $t \in \R_+ \mapsto \rho_t \in \mathcal{P}(\M)$, 
where $\mathcal{P}(\M)$ denotes the family of probability measures on $\M$.
However, we simply refer to the family $\{\rho_t\}_{t \in \R_+}$ of measures
as an evolution system of measures.} 
a family $\{\rho_t\}_{t \in \R_+}$ of probability measures on $\M$
is an evolution system of measures indexed by $\R_+$ if
\begin{align}
\int_\M  F(\varphi) \rho_{t_2}(d\varphi)= 
 \int_\M P_{t_1, t_2} F(\varphi) \rho_{t_1}(d\varphi)
 \label{evo1}
\end{align}

\noi
for any bounded continuous function $F$ on $\M$ and $t_2 \ge t_1 \ge 0$.
Note that \eqref{evo1} is equivalent to 
\[ \rho_{t_2} = P^*_{t_1, t_2} \rho_{t_1} \]

\noi
for any $t_2 \ge t_1 \ge 0$.
If there exists an invariant measure $\rho$, 
then by setting $\rho_t = \rho$, $t \in \R_+$, 
the family $\{\rho_t\}_{t\in \R_+}$ is obviously
an evolution system of measures.
It is in this sense that the notion of an evolution system
of measures is a generalization of the notion of an invariant measure.

Given $t \in \R_+$, let  $\mu_{1+t}$
be the white noise of variance $1+t$ defined in \eqref{mu1}.
Then, the following corollary follows from  \eqref{th1} and  the flow property 
\begin{align}
\Phi_{t_1, t_3} = \Phi_{t_2, t_3}\circ \Phi_{t_1, t_2}
\label{flow}
\end{align}

\noi
for 
$t_3 \ge t_2\ge t_1 \ge 0$ of the solution map to SKdV~\eqref{KdV1} constructed in 
Theorem~\ref{THM:1}.

\begin{corollary}\label{COR:x}
Let $\mu_{1+t}$ be the white noise measure with variance $1+t$ 
as in \eqref{mu1}.
Then, the family  $\{ \mu_{1+t}\}_{t\in \R_+}$
is an evolution system of measures
for SKdV \eqref{KdV1} with the white noise initial data $u_0^\o$ in~\eqref{series1}.

\end{corollary}

Furthermore, we have the following corollary to Theorem \ref{THM:1}.

\begin{corollary}\label{COR:2}
\textup{(i)} Given $\al \ge 0$, let $u_{0, \al}^\o$ be a white noise on $\T$ with variance $\al$
given by 
\begin{align*}
u_{0, \al}^\o(x) = \sqrt{\al}  \sum_{n \in \Z} g_n(\o) e^{inx}, 
\end{align*}

\noi
where $\{g_n \}_{n \in \Z}$ is as in \eqref{series1}.
Then, 
 with probability $1$, 
 there exists   a unique global-in-time solution
 $u$ to~\eqref{KdV1},
 with $u|_{t = 0} = u_{0, \al}^\o$.
   Moreover, for any $t\ge 0$, 
 we have
\begin{align}
\Law (u(t)) = \mu_{\al + t}, 
\label{mu1a}
\end{align}

\noi
where $\mu_{\al + t}$ is as in \eqref{mu1}.
Namely,   
 $u(t)$ is a white noise with variance $\al+t$.

\smallskip

\noi
\textup{(ii)} Let $w_0$ be a deterministic function in $L^2(\T)$ and $\al > 0$.
Then, 
 with probability $1$, 
 there exists   a unique global-in-time solution
 $u$ to~\eqref{KdV1}
 with $u|_{t = 0} = w_0 + \sqrt \al \, u_{0, \al}^\o$, 
where $u_0^\o$ is the white noise on $\T$ with variance $\al$ as in \eqref{series1}.

\end{corollary}

Part (i) of Corollary \ref{COR:2} directly follows from Theorem \ref{THM:1}
together with the flow property~\eqref{flow}
and 
 the time translation invariance (in law) of SKdV \eqref{KdV1}.
 See also Remark~\ref{REM:finite}.
Part (ii) of Corollary~\ref{COR:2} follows from Corollary \ref{COR:2}\,(i)
and the Cameron-Martin theorem~\cite{CM} by noting that $L^2(\T)$ is the Cameron-Martin space of $\mu_\al = \Law (\sqrt \al u_0^\o)$.
See \cite{OQ} for a further discussion.

\medskip

Thanks to the time reversibility of the KdV equation, 
Theorem \ref{THM:1} and Corollary \ref{COR:2}
also hold for negative times (where the variances $1+t$ in \eqref{th1}
and $\al+t$ in \eqref{mu1a}
are replaced by $1+|t|$ and $\al+|t|$, respectively.
For simplicity of the presentation, however, we only consider positive times in the remaining part 
of the paper.
Moreover, in the following discussion, 
in considering a stochastic flow on a time interval $[t_1, t_2]$, 
it is understood that random initial data at time $t_1$
 and a stochastic forcing on $[t_1, t_2]$ are independent
(which is justified by \eqref{white1}).


\subsection{Outline of the proof}

Let us now describe some aspects of the proof of Theorem~\ref{THM:1}.
Except in the small data regime (including 
a small perturbation of a known global solution), 
one usually needs to exploit conservation laws in order to construct global-in-time solutions
to nonlinear dispersive PDEs.
A remarkable intuition by Bourgain in \cite{BO2}
was to use (formal) invariance of a Gibbs measure 
as a replacement of a conservation law
to construct global-in-time solutions with the Gibbsian initial data.
More precisely, he used the rigorous invariance of the truncated Gibbs measures
for the associated truncated dynamics
and combined it with a PDE approximation argument to 
construct the desired global-in-time invariant Gibbs dynamics.
This argument, known as Bourgain's invariant measure argument,  has been applied to many dispersive PDEs
with random initial data (and stochastic forcing), in particular
over the last fifteen years.
See the survey papers \cite{OH6, BOP4, Tz1}  for a further discussion on this topic and the references therein.
See also 
\cite{GKOT, OOT1, OOT2} for more recent results in the context 
of stochastic dispersive PDEs.
We point out that Bourgain's invariant measure argument has also been 
applied
to globalize solutions to stochastic parabolic PDEs; see, 
for example, 
\cite{HM, ORW, ORTW}.

In the current problem at hand, due to the lack of a damping term, 
there is no invariant measure for SKdV \eqref{KdV1}, 
and thus Bourgain's invariant measure argument is not applicable.
It is, however, easy to see, at a formal level, (as explained below)
that SKdV \eqref{KdV1} with the white noise initial data \eqref{series1} possesses
a (formal) evolution system of measures $\{ \mu_{1+t}\}_{t \in \R_+}$, 
where $\mu_{1+t}$ is a white noise measure with variance $1+t$
defined in \eqref{mu1}.
See also Proposition~\ref{PROP:finite}.
Our main strategy is then to {\it use this (formal) evolution system of measures
$\{ \mu_{1+t}\}_{t \in \R_+}$ 
as a replacement of a (formal) invariant measure
in Bourgain's invariant measure argument} (and hence as a replacement of a conservation law
in the deterministic setting).

Before proceeding further, 
let us provide a heuristic argument for the claim 
that $\{\mu_{1+t}\}_{t \in \R_+}$ is an evolution system of measures
for SKdV \eqref{KdV1} with the white noise initial data.
First, 
view the SKdV dynamics \eqref{KdV1}
as a superposition of the deterministic KdV~\eqref{KDV}
and 
\begin{align}
\dt  u = \xi
\label{KdV3a}
\end{align}

\noi
(at the level of infinitesimal generators).
On the one hand, the white noise (with any variance) is known 
to be invariant under the flow of the deterministic KdV \eqref{KDV};
see  \cite{QV, OH1, OH6, OQV, KMV}.
On the other hand, 
the stochastic flow \eqref{KdV3a} with a white noise initial data (with any variance)
increases the variance by the length of the time interval under consideration.
Then, the claim follows, at least at a purely formal level, 
from these observations 
 together with 
the 
Lie-Trotter product formula \cite[Section VIII.8]{RS}:
\begin{align}
e^{t (A+ B)} = \lim_{n \to \infty} \big[ e^{\frac tn A}e^{\frac tn B}\big]^n
\label{tro}
\end{align}

\noi
(which holds, for example, 
for finite-dimensional matrices $A, B$).
We point out that 
the Lie-Trotter product formula \eqref{tro}
is not directly applicable to our problem, 
and the core of the proof of Theorem \ref{THM:1}
consists of justifying this heuristic argument
by  an approximation argument,
which we explain next.

\medskip

\noi
$\bullet$ {\bf Truncated SKdV dynamics.}
Given $N \in \NB$, 
let 
$\P_{N}$ denotes the Dirichlet projection on (spatial) frequencies $\{|n|\le N\}$.
Then, consider the following truncated SKdV equation: 
\begin{align}
\begin{cases}
\dt u^N + \dx^3 u^N + \P_{ N}(\P_{ N} u^N \cdot \dx \P_{ N} u^N) =  \xi\\
u^N|_{t = 0} =  u_0^\o, 
\end{cases}
\label{KdV4}
\end{align}

\noi
where $u_0^\o$ is the white noise given in \eqref{series1}.
Note that the truncation appears only on the nonlinearity, but not
on the noise or the initial data.
With $\P_{N}^\perp = \Id - \P_{ N}$, 
set 
\begin{align*}
 u_N = \P_{N} u^N
 \qquad \text{and} \qquad u_N^\perp = \P_N^\perp u^N.
\end{align*}

\noi
 Then, the truncated SKdV dynamics \eqref{KdV4}
 decouples into 
 the finite-dimensional nonlinear dynamics for the low frequency part $u_N = \P_{N} u^N$:
\begin{align}
\begin{cases}
\dt u_N + \dx^3 u_N + \P_{N}(u_N \dx u_N) = \P_{N} \xi\\
u_N|_{t = 0} = \P_{N} u_0^\o, 
\end{cases}
\label{KdV5}
\end{align}

\noi
and the linear dynamics for the high frequency part 
$u_N^\perp = \P_N^\perp u^N$:
\begin{align}
\begin{cases}
\dt u_N^\perp + \dx^3 u_N^\perp  = \P_{N}^\perp \xi\\
u_N^\perp |_{t = 0} = \P_{N}^\perp u_0^\o.
\end{cases}
\label{KdV6}
\end{align}

\noi
It is easy to see that both \eqref{KdV5} and \eqref{KdV6}
are globally well-posed (which implies that \eqref{KdV4} is globally well-posed); see Section \ref{SEC:finite}.
For $t_2 \ge t_1 \ge 0$,
we denote by 
 $\Phi^{N, \low}_{t_1, t_2}$
 and $\Phi^{N, \high}_{t_1, t_2}$  
 the solution maps for \eqref{KdV5} and \eqref{KdV6}
sending data $\varphi$  at time $t_1$ to the solutions 
$\Phi^{N, \low}_{t_1, t_2}\varphi$
and $\Phi^{N, \high}_{t_1, t_2}\varphi$ at time $t_2$.
We let $P^{N,\low}_{t_1, t_2}$
and $P^{N,\high}_{t_1, t_2}$ denote the transition semigroups
for \eqref{KdV5} and \eqref{KdV6}, respectively, 
defined as in \eqref{trans1}, where the expectation 
is taken over the noise restricted to the time interval $[t_1, t_2]$.
We also use 
$\Phi^{N}_{t_1, t_2}$
and $P^{N}_{t_1, t_2}$
to denote the solution map and the transition semigroup for the truncated SKdV \eqref{KdV4}.

Given $\al \ge 0$, let $\mu_\al$ be the white noise measure (with variance $\al$)
as in \eqref{mu1}.
Then, we can write $\mu_\al$ as
\begin{align}
\begin{split}
 \mu_\al 
 & =  \mu_\al^{N, \low} \otimes \mu_\al^{N, \high}\\
&  = 
 (\P_N)_* \mu_\al \otimes (\P_N^\perp)_* \mu_\al, 
\end{split}
 \label{mu2}
\end{align}

\noi
where $\mu_\al^{N, \low} = (\P_N)_* \mu_\al$ and 
$\mu_\al^{N, \high} = (\P_N^\perp)_* \mu_\al$
the pushforward image measures of $\mu_\al$ under
$\P_N$ and $\P_N^\perp$, respectively.
Note that 
$\mu_\al^{N, \low}$ and $\mu_\al^{N, \high}$
are nothing but the white noise measures (with variance $\al$) on $E_N = \text{span}\{e^{inx}: |n| \le N\}$
and $E_N^\perp = \text{span}\{e^{inx}: |n| >  N\}$, 
respectively, 
where the latter span is taken over the
space $\D'(\T)$ of distributions on $\T$.

The high frequency dynamics \eqref{KdV5}
is linear and it is easy to verify 
that 
\begin{align}
(P^{N, \high}_{t_1, t_2})^* \mu^{N, \high}_{1+t_1}
=   \mu^{N, \high}_{1+t_2 }.
\label{mu3}
\end{align}

\noi
By writing it on the Fourier side, 
we see that the low frequency dynamics \eqref{KdV5}
is nothing but a finite-dimensional system of SDEs,
which can be viewed as the superposition of 
the finite-dimensional KdV dynamics:
\begin{align}
\dt u_N + \dx^3 u_N + \P_{N}(u_N \dx u_N) =0
\label{KdV7}
\end{align}

\noi
and the linear stochastic dynamics:
\begin{align}
\dt u_N = \P_{N} \xi.
\label{KdV8}
\end{align}

\noi
While the former \eqref{KdV7} preserves
the white noise $\mu_\al^{N, \low}$ (with any variance), 
the latter~\eqref{KdV8}
increases the variance of the white noise initial data by the length of the time interval under consideration.
Then, in view of the Lie-Trotter product formula \eqref{tro}, 
we see that 
\begin{align}
(P^{N, \low}_{t_1, t_2})^* \mu^{N, \low}_{1+t_1}
=   \mu^{N, \low}_{1+t_2 }.
\label{mu4}
\end{align}

\noi
Putting \eqref{mu3} and \eqref{mu4} together, 
we then obtain the following proposition.

\begin{proposition}\label{PROP:finite}
Let $N \in \NB$. 
Then, for any $t_2 \ge t_1 \ge 0$, we have 
\begin{align*}
(P^{N}_{t_1, t_2})^* \mu_{1+t_1}
=   \mu_{1+t_2 }, 
\end{align*}

\noi
where 
  $P^{N}_{t_1, t_2}$
is  the transition semigroup for the truncated SKdV \eqref{KdV4}.
Namely, $\{\mu_{1+t}\}_{t \in \R_+}$
is an evolution system of measures
for the truncated SKdV \eqref{KdV4}.
\end{proposition}

We present the proof of Proposition \ref{PROP:finite} in 
Section \ref{SEC:finite}.
As for the low frequency part of the claim, 
instead of decomposing the low frequency dynamics \eqref{KdV5} into \eqref{KdV7}
and  \eqref{KdV8}
and applying
the Lie-Trotter product formula \eqref{tro},
we verify  \eqref{mu4}
by 
directly showing that $\mu^{N, \low}_{1+t}$
is the unique solution to 
the Kolmogorov forward equation
(= the Fokker-Planck equation).

\begin{remark}\label{REM:finite} \rm
Let $\al \ge 0$. A straightforward modification of the proof of Proposition \ref{PROP:finite}
yields
\begin{align*}
(P^{N}_{t_1, t_2})^* \mu_{\al}
=   \mu_{\al + (t_2-t_1) }, 
\end{align*}

\noi
which is the key ingredient for proving Corollary \ref{COR:2}\,(i), 
replacing Proposition \ref{PROP:finite}.

\end{remark}

Once we obtain 
Proposition \ref{PROP:finite}, 
we use ideas from Bourgain's invariant measure argument \cite{BO2}
together with the nonlinear analysis in \cite{OH4},
and establish a probabilistic uniform (in $N$) growth bound 
on the solutions to the truncated SKdV \eqref{KdV4}.
See Proposition~\ref{PROP:main}.
Finally, 
Theorem \ref{THM:1}
follows from a PDE approximation argument
and this
probabilistic uniform  growth bound.
See Section \ref{SEC:5}.

\medskip

\smallskip

\noi
$\bullet$ {\bf Mean-zero assumption:}
Recall that 
 the bilinear estimate  \eqref{bilin1}
 holds only for 
 (spatial) mean-zero functions, 
 namely, 
 the spatial means
  of $u(t)$ and $v(t)$ are zero for any $t \in \R$.
In the case of the deterministic KdV \eqref{KDV}, 
if initial data $u_0$ has non-zero mean $\al_0$, 
then 
the following  Galilean transformation:
\[u(x, t) \longmapsto u(x +\alpha_0 t, t) - \alpha_0\]

\noi
as in \cite{CKSTT5} together with the conservation of the (spatial) mean
under KdV
 transforms KdV with a non-zero mean
into the mean-zero KdV (so that the bilinear estimate \eqref{bilin1}
is applicable).
In the case of SKdV with an additive noise, 
the spatial mean of a solution is no longer conserved.
Nonetheless, 
in \cite{ddt-1, OH4}, 
a similar transformation was employed to reduce  SKdV  with 
an additive noise
to the mean-zero case.
The transformation in this case 
depends not only on the mean of the initial condition
but also on the Brownian motion $\beta_0$ at the zeroth frequency
in \eqref{W1}.
See \cite{ddt-1, OH4} for details.

For conciseness of the presentation, 
we impose the following mean-zero assumption in the remaining part of the paper.

\begin{itemize}
\item We assume that the white noise initial data $u_0^\o$ in \eqref{series1}
and the space-time white noise $\xi$ in \eqref{KdV1} and \eqref{KdV4}
have spatial mean-zero.
This means that the random initial data is now given by 
\begin{align}
u_0^\o(x) = \sum_{n \in \Z_*} g_n(\o) e^{inx}, 
\label{series4}
\end{align}

\noi
where $\Z_* = \Z\setminus \{0\}$, 
and the stochastic forcing $\xi$ is given by 
the distributional time derivative of 
\begin{align}
W(t)
 = \sum_{n \in \Z_*} \be_n (t) e^{inx}.
\label{W2}
\end{align}

\noi
Namely, we have $\xi = \P_{\ne 0}\xi$, 
where $\P_{\ne 0}$ is the projection onto
the non-zero (spatial) frequencies.
This assumption together with the presence of the derivative on
the nonlinearity $u \dx u = \frac 12 \dx u^2$
implies that 
a solution $u$ to SKdV \eqref{KdV1}
has spatial mean zero as long as it exists.

\end{itemize}

\smallskip

\noi
It is understood that all the functions/distributions
have spatial mean zero in the following.
The required modifications 
to handle the general case (i.e.~with the white noise $u_0^\o$
in \eqref{series1} and the space-time white noise $\xi$ without
the projection $\P_{\ne 0}$) 
are straightforward and hence we omit details.
See \cite{OH4} for details.

\medskip

We conclude this introduction by stating several remarks.

\begin{remark}\rm
The usual application of Bourgain's invariant measure argument
provides a   growth bound\footnote{At least in the setting of \cite{BO2}.
In the singular setting, we have a growth bound
by a suitable power of $\log t$.
See, for example,  Section 5 in \cite{ORTz}.}
 on a solution
by  $\sqrt{\log t}$ for $t \gg1$, where the implicit constant is random.
In the current SKdV problem, 
we instead obtain a growth bound on a solution 
by (something slightly faster than)
$\sqrt{ t\log t}$ for $t \gg1$, where the extra factor $\sqrt t$
comes from the fact that the variance of the white noise at time $t$ grows like $\sim t$.
 See Remark \ref{REM:bound}.

\end{remark}

\begin{remark}\rm
(i) 
As mentioned above, 
by applying the $I$-method, Colliander, Keel, Staffilani, Takaoka, and Tao 
\cite{CKSTT4} proved global well-posedness 
of  the deterministic KdV \eqref{KDV}
in $H^{-\frac 12}(\T)$.
It would be of interest to apply the $I$-method to 
study global well-posedness of SKdV \eqref{KdV1}
with general deterministic initial data.
In \cite{CLO}, the first author with Cheung and Li 
adapted the $I$-method to the stochastic setting
and proved global well-posedness, below the energy space, of 
the stochastic nonlinear Schr\"odinger equation (SNLS) on $\R^3$
 with additive stochastic forcing, white in time and correlated in space.
On the one hand, the $I$-method is suitable 
for controlling an $L^2$-based Sobolev norm.
On the other hand, 
 the only known local well-posedness result
of SKdV \eqref{KdV1} is in the Fourier-Besov space
$\ft b^s_{p, \infty}$ (at this point), and thus there is a non-trivial difficulty
in adapting the $I$-method to this problem.

\medskip

\noi
(ii) In \cite{KVZ}, Killip, Vi\c{s}an, and Zhang
exploited the complete integrable structure
of the deterministic KdV \eqref{KDV}
and established a global-in-time  a priori bound
for solutions to KdV in $H^s(\T)$, $s \ge -1$.
This a priori bound was given by 
a sum of suitable rescaled
perturbation determinants,
(each of which is given as an infinite series).
It would also be of interest to investigate
if their approach can be adapted to the current stochastic setting
(and moreover to the Fourier-Besov setting, using the ideas in \cite{OW2}).

\end{remark}

\begin{remark}\rm

Consider the following SNLS
on~$\T$:
\begin{align}
i \dt u - \dx^2 u + |u|^2 u  = \xi,
\label{NLS1}
\end{align}

\noi
where $\xi$ is a complex-valued space-time white  noise on $\T\times \R_+$, 
with the complex-valued white noise initial data:
\begin{align}
u_0^\o(x) = \sum_{n \in \Z} g_n(\o) e^{inx}, 
\label{series3}
\end{align}

\noi
where $\{g_n \}_{n \in \Z}$ is a family of independent standard 
complex-valued Gaussian random variables.
(Here, we do not impose the condition 
$g_{-n} = \cj{g_n}$.)
Due to the low regularity of the initial data and the forcing, 
we need to renormalize the nonlinearity in \eqref{NLS1}
to that considered in \cite{Christ1, GO, FOW, OW}.
In the following discussion, 
we suppress this renormalization issue.

Let us first consider the (deterministic)  nonlinear Schr\"odinger equation (NLS)
on~$\T$:
\begin{align}
i \dt u - \dx^2 u + |u|^2 u  = 0.
\label{NLS2}
\end{align}

\noi
Given $\al > 0$, let $\mu_\al = \Law(\sqrt \al\, u_0^\o)$ 
with $u_0^\o$ as in \eqref{series3} be
the (complex) white noise measure with variance $\al$.
Formally, we have
\begin{align*}
d \mu_\al = Z_\al^{-1} e^{-\frac 1{2\al} \int_\T |u|^2 dx } du.
\end{align*}

\noi
See \cite{OH6, OQV}.
Then, in view of the conservation of the $L^2$-norm under \eqref{NLS2}
and the fact that NLS \eqref{NLS2} is Hamiltonian, 
we expect that the white noise measure $\mu_\al$ is invariant under the NLS dynamics.
In \cite{OQV}, 
the first two authors with Valk\'o
proved formal invariance of the white noise measure under NLS \eqref{NLS2}
in the sense that the white noise measure is 
a weak limit of invariant measures for NLS \eqref{NLS2}.
In the same paper, 
they also conjectured invariance of the white noise under NLS \eqref{NLS2}.
This conjecture remains as a challenging open problem to date, in particular
due to the critical nature of the well-posedness issue
for~\eqref{NLS2} (and also for~\eqref{NLS1}) with white noise initial data;
see \cite{FOW, DNY2}.
See also \cite{OTzW}
for invariance of the white noise measure
under the fourth order NLS on $\T$, 
where $-\dx^2$ in~\eqref{NLS2} is replaced by $(-\dx^2)^2$.

Let us now turn our attention to SNLS \eqref{NLS1}.
As in the SKdV case, by viewing \eqref{NLS1}
as a superposition of the deterministic NLS \eqref{NLS2}
and the stochastic flow $i \dt u = \xi$
together with the conjectural invariance of the white noise under NLS \eqref{NLS2}, 
we arrive at the following conjecture.

\begin{conjecture}\label{CONJ:3}
The family  $\{ \mu_{1+t}\}_{t\in \R_+}$ of the white noise measures with variance $1+ t$
is an evolution system of measures
for SNLS \eqref{NLS1} with the white noise initial data $u_0^\o$ in~\eqref{series3}.

\end{conjecture}

This conjecture is of importance not only from the viewpoint
of mathematical analysis but also from the viewpoint of applications
due to the importance of SNLS \eqref{NLS1} (and NLS~\eqref{NLS2})
in nonlinear fiber optics.
A straightforward modification of the proof of Proposition \ref{PROP:finite}
shows that, 
for any $N \in \NB$, 
the family $\{ \mu_{1+t}\}_{t\in \R_+}$
is an evolution system of measure 
for the following truncated SNLS:
\begin{align*}
i \dt u^N - \dx^2 u^N + \P_N(|\P_N u^N|^2 \P_Nu^N)  = \xi
\end{align*}

\noi
with the white noise initial data $u_0^\o$ in~\eqref{series3}.
The main obstacle for proving Conjecture~\ref{CONJ:3}
is the local well-posedness issue as in the case of NLS \eqref{NLS2}
with the white noise initial data.

\end{remark}

\section{Finite-dimensional approximations and their distributions}
\label{SEC:finite}

In the remaining part of the paper, we work on a probability space $(\Omega,\mathcal{F},\PP)$ supporting
\begin{itemize}
\item A family $\{g_n\}_{n \in \NB}$ of  independent standard complex-valued Gaussian random variables: 
\begin{equation}
g_n=\Re g_n + i\Im g_n,\qquad n \in \NB.
\label{gauss1}
\end{equation}

\noi
Here,  $\{\Re g_n, \Im g_n\}_{n\in \NB}$ is a
family of independent real-valued Gaussian random variables with mean 0 and variance $\frac 12$.
We then set  $g_{-n}=\cj{g_n}$, $n \in \NB$.
 The random variables $g_n$ are used to define the spatial white noise $u_0^\o$ 
 on $\mathbb{T}$ in \eqref{series4} which we use as initial data
 for \eqref{KdV1} and \eqref{KdV4}.

\smallskip
\item A family $\{\beta_n\}_{n \in \NB}$ of independent complex-valued Brownian motions,
satisfying~\eqref{W1a}:
\[\beta_n(t)=\Re \beta_n(t)+i \Im \beta_n(t),\qquad n \in \NB, \]

\noi
which is also independent of $\{g_n\}_{n\in \NB}$. 
We then set  $\beta_{-n}=\cj {\beta_{n}}$, $n \in \NB$.
 The Brownian motions $\beta_n$ serve to define the driving space-time white noise appearing in 
 \eqref{KdV1} as well as its truncated version \eqref{KdV4}.
\end{itemize}

\medskip

\noi
We emphasize that 
we only work with (spatial) mean-zero functions/distributions in the following.
 Given $\al \ge 0$, let $u_{0, \al}^\o$ be a white noise on $\T$ with variance $\al$
given by 
\begin{align}
u_{0, \al}^\o(x) = \sqrt \al\,  u_0^\o(x) = \sqrt{\al}  \sum_{n \in \Z_*} g_n(\o) e^{inx}, 
\label{series5}
\end{align}

\noi
where $\{g_n \}_{n \in \Z}$ is as in \eqref{gauss1}, 
and set  
\begin{align}
\mu_\al = \Law (u_{0, \al}^\o)
\label{mu7}
\end{align}
to be the (mean-zero) white noise measure with variance $\al$.
With this version of  $\mu_\al$, 
we set
\[ \mu_\al^{N, \low} = (\P_N)_* \mu_\al
\qquad \text{and}\qquad
\mu_\al^{N, \high} = (\P_N^\perp)_* \mu_\al.\]

\noi
Then, \eqref{mu2} holds in the current setting.

\medskip

In this section, we study the truncated SKdV \eqref{KdV4}
%
and present the proof of Proposition~\ref{PROP:finite}.
In view of the discussion in Section \ref{SEC:1}, 
it suffices to prove \eqref{mu3} and \eqref{mu4}
for the high and low frequency dynamics, respectively.

We first consider the high frequency dynamics \eqref{KdV6}:
\begin{align}
\dt u_N^\perp + \dx^3 u_N^\perp  = \P_{N}^\perp \xi.
\label{KdV9a}
\end{align}

\noi
By working  on the Fourier side, 
we see that \eqref{KdV9a}
is a system of decoupled linear SDEs for each frequency.
In particular, \eqref{KdV9a} is globally well-posed
and the  solution to \eqref{KdV9a} is given by 
\begin{align*}
\ft {u^\perp_N}(n, t) = e^{i t n^3} \ft {u^\perp_N}(n, 0) + \int_0^t
e^{i (t-t') n^3} d\be_n(t'), \qquad |n| > N, 
\end{align*}
 
\noi
for general initial data $u_N^\perp(0) = \P_N^\perp u_N^\perp(0)$.
In particular, when the initial data is given by $\P_N^\perp u_0^\o$ with $ u_0^\o$ in \eqref{series4}, 
we have 
\begin{align*}
\ft {u^\perp_N}(n, t) = e^{i t n^3} g_n + \int_0^t
e^{i (t-t') n^3} d\be_n(t') =: \1_n + \II_n, \qquad |n| > N, 
\end{align*}

\noi
Note that $\Law(\1_n) = \Law(g_n)$
(see Lemma 4.2 in \cite{OTz})
and $\Law(\II_n) = \Law(\sqrt t \, g_n)$.
Then, from the independence of $\1_n$ and $\II_n$, 
we conclude that 
\begin{align}
(P^{N, \high}_{0, t})^* \mu^{N, \high}_{1}
=   \mu^{N, \high}_{1+t }.
\label{mu8}
\end{align}

\noi
Therefore, from \eqref{mu8} and the flow property 
of the solution map 
 $\Phi^{N, \high}_{t_1, t_2}$ for \eqref{KdV9a}
 (analogous to \eqref{flow}), 
 we conclude 
\eqref{mu3}.

\medskip

Let us now turn our attention to the low frequency dynamics \eqref{KdV5}:
\begin{align}
\dt u_N + \dx^3 u_N + \P_{N}(u_N \dx u_N) = \P_{N} \xi.
\label{KdV10}
\end{align}


\begin{lemma}\label{LEM:LWP}
Let $n \in \NB$.
Given any initial data $u_N(0) = \P_N u_N(0)$ with $\ft u_N(0, 0) = 0$, 
there exists a unique global solution  $u_N\in C(\R_+;L^2(\mathbb{T}))$
to  \eqref{KdV10} with $u_N|_{t = 0} = u_N(0)$.
\end{lemma}

\begin{proof}
By writing \eqref{KdV10} in the Duhamel formulation, we have
\begin{equation}
u_N(t)=S(t)u_N(0)-\frac{1}{2}\int_0^t S(t-t')\partial_x \P_N (u_N^2)(t')d t'
+\int_0^t S(t-t')d \P_N W(t'), 
\label{KdVN}
\end{equation}

\noi
where $W$ is as in \eqref{W2}.
Note that $u_N(t) = \P_N u_N(t) $ as long as the solution $u_N$ exists.
By Bernstein's inequality (\cite[Appendix A]{Tao}), we have
\begin{align}
\|\partial_x \P_N u_N^2\|_{C([0, T]; L^2_x)}
\les N \| u_N\|_{C([0, T]; L^4_x)}^2
\les N^\frac 32  \| u_N\|_{C([0, T]; L^2_x)}^2, 
\label{a1}
\end{align}

\noi
which allows us to control the second term on the right-hand side of \eqref{KdVN}.
By the unitarity of $S(t)$ on $L^2(\T)$
and the basic property of a Wiener integral, we have
\begin{align*}
\E\Bigg[\bigg\|\int_0^t S(t-t')d \P_N W(t')\bigg\|_{C([0, T]; L^2_x)}^2\Bigg]
\les T N.
\end{align*}

\noi
In particular, we have 
\begin{align}
\bigg\|\int_0^t S(t-t')d \P_N W(t')\bigg\|_{C([0, T]; L^2_x)}
\le C(\o) T^\frac 12 N^\frac 12
\label{a2}
\end{align}

\noi
for some almost surely finite random constant $C(\o) > 0$.
Hence, we conclude from a standard contraction argument 
in $C([0, T]; L^2(\T))$ with \eqref{a1} and \eqref{a2}
that \eqref{KdV10} is locally well-posed.
Furthermore, the solution exists globally in time 
as long as its $L^2(\T)$-norm remains bounded.

As observed in \cite[Theorem 1.5]{ddt-1} and \cite[Section 3.2]{ddt-2}, a simple argument,
 using Ito's formula, Doob's martingale inequality, and the $L^2$-conservation of the  truncated KdV equation~\eqref{KdV7}, provides 
the following bound:
\begin{align*} \E\Big[\sup_{t \in [0, T]}\|u_N(t)\|_{L^2}^2 \Big]
& \le \|u_N(0)\|_{L^2}^2 
+ C(T) \| \P_N \|_{\HS(L^2; L^2)}\\
& \le \|u_N(0)\|_{L^2}^2 
+ C'(T) N^\frac 12
\end{align*}
\noi
for any finite $T > 0$, 
where
$\|\cdot \|_{\HS(L^2; L^2)}$ denotes the Hilbert-Schmidt norm from $L^2(\T)$ to $L^2(\T)$.
From this a priori bound,  we conclude global well-posedness of \eqref{KdV10}.
\end{proof}

%

In the following, 
we study the evolution of the distribution of
the solution  $u_N(t)$
to  the low frequency dynamics \eqref{KdV10}.
Let $p_n(t) =\Re  \ft u_N(n, t)$
and $q_n(t) = \Im \ft u_N(n, t)$ for $1\le  |n| \le N$.
Since $u_N$ is real-valued,  we have 
\begin{align*}
p_{-n} = p_n \qquad \text{and} \qquad q_{-n} = - q_n.
\end{align*}

\noi
Then, by writing \eqref{KdV10} on the Fourier side, 
we obtain the following finite-dimensional system of SDEs
for $(\bar p, \bar q) = (p_1, \dots, p_N, q_1, \dots, q_N)$:
\begin{align}
\begin{split}
d p_n & = P_n dt +  d (\Re \be_n),  \\
d q_n & = Q_{n}dt + d (\Im \be_n) 
\end{split}
\label{pq2}
\end{align}

\noi
for $n = 1, \dots,  N$, 
where $P_n$ and $Q_n$ are defined by 
\begin{align}
\begin{split}
P_n & :  = -n^3 q_n + \sum_{\substack{n = n_1 + n_2\\1\le |n_1|, |n_2| \le N}} n_2
(p_{n_1}q_{n_2} + q_{n_1} p_{n_2}) ,  \\
Q_{n} & :  = n^3 p_n - \sum_{\substack{n = n_1 + n_2\\1 \le |n_1|, |n_2| \le N}} n_2 
(p_{n_1}p_{n_2} - q_{n_1} q_{n_2}). 
\end{split}
\label{pq3}
\end{align}

\noi
Define  
$A(\bar p , \bar q) = A(p_1, \dots, p_N, q_1, \dots, q_N)$
by 
\begin{align}
A(\bar p , \bar q) = (P_1, \dots, P_N, Q_1, \dots, Q_N).
\label{ZX2}
\end{align}

\noi
Then, we have
\begin{align}
\text{div}_{\bar p, \bar q}  A(\bar p , \bar q) 
= \sum_{n = 1}^N
(\dd_{p_n} P_n + \dd_{q_n} Q_n)
= \sum_{n = 1}^N \ind_{2n \le N} (nq_{2n} -nq_{2n} ) = 0.
\label{pq5}
\end{align}

Let $\bar x = (x_1, \dots, x_{2N}) = (p_1, \dots, p_N, q_1, \dots, q_N)$.
In the following, we briefly go over the derivation of the Kolmogorov forward equation for the evolution of the density of the distribution for 
\begin{align}
\ft U(t) = \big(\Re \ft u_N(1, t), \dots, \Re \ft u_N(n, t),
\Im \ft u_N(1, t),\dots, \Im \ft u_N(n, t)\big).
\label{pq55}
\end{align}

\noi
See, for example,  \cite{stroock, daprato}.
Recalling from  \eqref{W1a} that $\E[(\Re \be_n(t)^2] = \E[(\Im \be_n(t)^2] = \frac t2$, 
we see that 
the Kolmogorov operator $\L$ for \eqref{pq2} is given  by 
\begin{align}
\L = \frac{1}{4}\Delta_{\bar x} +A(\bar x)\cdot \nabla_{\bar x}, 
\label{pq5a}
\end{align}

\noi
where $A(\bar x)$ is given by 
\begin{align}
A(\bar x) = (P_1, \dots, P_N, Q_1, \dots, Q_N).
\label{pq56}
\end{align}

\begin{lemma}
Let $f_0(\bar x)$ be a density of the distribution for $\ft U(0)$.
Then, the  density $f(\bar x, t)$ of the distribution for $\ft U(t)$ 
satisfies 
 the following Kolmogorov forward equation on $\R^{2N}$\textup{:}
\begin{equation}
 \begin{cases}
 \partial_t f(\bar x, t)
 -\frac{1}{4}\Delta_{\bar x}f(\bar x, t)
 +A(\bar x)\cdot \nabla_{\bar x} f(\bar x, t)=0,\\
  f|_{t = 0}=f_0,
\end{cases}
\label{IVP}
\end{equation}

\noi
where  the vector field $A(\bar x)$ is as in \eqref{pq56}.
\end{lemma}

\begin{proof}

This is classical, so we only provide a sketch. 
Consider 
\begin{align}
\begin{cases}
(\partial_t - \L)g(\bar x, t)=0\\ 
\, g|_{t = 0} = g_0, 
\end{cases}
\label{pq6}
\end{align}

\noi
where $\L$ is as in \eqref{pq5a}.
It is well known that \eqref{pq6}
has a smooth fundamental solution $p(\bar x,\bar y, t)$ for $(\bar x,\bar y, t)\in  \mathbb{R}^{2N}\times \mathbb{R}^{2N}
\times \R_+$, 
and thus,  
for initial data $g_0\in C^2(\mathbb{R}^{2N})$ with bounded derivatives, 
the unique solution to \eqref{pq6} is given by
\begin{equation}
g(\bar x, t)=\int_{\R^{2N}}  g_0( \bar y)p(\bar x, \bar y, t)d \bar y.
\label{pq7}
\end{equation}

\noi
Here, $(\bar x, t)\mapsto p(\bar x, \bar y, t)$ 
satisfies $(\partial_t - \L)p(\bar x,\bar y, t)=0$ for each fixed $\bar y \in \R^{2N}$. 
See,  for example, \cite[Lemma 3.3.3]{stroock}. 
Moreover, $g(\bar x, t)$ has the following probabilistic representation
(\cite[Theorem 9.16]{daprato}):
\begin{align}
\begin{split}
g(\bar x, t)=\E_{\bar y = \ft U(t)} \big[ & g_0(\bar y)\, |  \, \bar x = \ft U(0)\big], 
\end{split}
\label{pq8}
\end{align}

\noi
where $\ft U(t)$ is as in \eqref{pq55}
and 
 the expectation on the right-hand side is taken  with respect to the vector 
 $\bar y = \ft U(t)$
 conditioned that  $\bar x = \ft U(0)$.
Hence, it follows from \eqref{pq8} and \eqref{pq7} that 
\begin{align*}
\E_{\bar y = \ft U(t)} [ g_0(\bar y)]
& = \E_{\bar x= \ft U(0)}\big[\E_{\bar y = \ft U(t)} [ g_0(\bar y)\mid \bar x = \ft U(0)]\big]\\
& = \int_{\R^{2N}}\int_{\R^{2N}}  g( \bar y)p(\bar x, \bar y, t)d \bar y f_0(\bar x) d \bar x.
\end{align*}

\noi
Therefore, 
the density $f(\bar y, t)$ of $\ft U(t) $ is given by 
 \begin{align}
 f(\bar y, t)=\int f_0(\bar x )p(\bar x,\bar y, t)d \bar x
 \label{pq9}
 \end{align}

Now, 
it follows from \eqref{pq7}, \eqref{pq8},  and
 Ito's formula (see, for example, the proof of Proposition 9.9 in \cite{daprato}), we see
\begin{align*}
 \int_{\R^{2N}}  g_0( \bar y)\dt p(\bar x, \bar y, t)d \bar y
& = \frac d {dt} \E_{\bar y = \ft U(t)} \big[  g_0(\bar y)\, |  \, \bar x = \ft U(0)\big]\\
& =\int (\L g_0)(\bar y) p(\bar x,\bar y, t)d\bar y\\
& =\int g_0(\bar y)\,(\L^t_{\bar y}p)(\bar x,\bar y, t)d\bar y, 
\end{align*}

\noi
where 
 $\L^t$ is the formal adjoint of $\L$ given by 
\[\L^t_{\bar y}=\frac{1}{4}\Delta_{\bar y}-A(\bar y)\cdot \nabla_{\bar y}.\]

\noi
Note that, in the computation of $\L^t$, we used \eqref{pq5}:
$\text{div}_{\bar y}  A(\bar y) = 0$. 
Hence, we conclude that  $(\bar y, t)\mapsto p(\bar x,\bar y, t)$ satisfies  
$(\partial_t-\L_{\bar y}^t)p(\bar x,\bar y, t)=0$, and 
therefore, we conclude from~\eqref{pq9} that 
 $f(\bar y, t)$ satisfies 
 $(\partial_t-\L_{\bar y}^t)f(\bar y, t) =0$.
\end{proof}

We are now ready to prove \eqref{mu3}.
Let $\g_\al$ be the density for the normal distribution
on~$\R$
with mean~0 and variance $\frac \al 2> 0$:
\begin{align*}
\gamma_\al(x)&=\frac{1}{\sqrt{\pi\al}}e^{-\frac{x^2}{\al}}.
\end{align*}

\noi
Then, in the current setting, the density 
of the distribution for $\P_N u_0^\o$ with $u_0^\o$ as in \eqref{series4}
is given by 
\begin{equation*}
f_0(\bar x)=\prod_{n= 1}^{2N} \gamma_1(x_n).
\end{equation*}

\noi
The following lemma shows that 
 the solution $u_N(t)$ to \eqref{KdV10} with initial data $u_{0, \al}^\o = \sqrt \al u_0^\o$ in~\eqref{series5}
 is distributed by the (mean-zero) white noise measure $\mu_{\al+t}$ in \eqref{mu7},
 which in particular proves \eqref{mu3}.
 
\begin{lemma}\label{LEM:stat}
For any $\al>0$, the function $f_{N, \al}$ given by 
\[f_{N, \al}(\bar x, t)=\prod_{n=1}^{2N} \gamma_{\al+t}(x_n)
=\frac{1}{(\pi (\al+t))^{\frac N2}} e^{-\frac{|\bar x|^2}{\al+t}}\] 
is the unique solution to \eqref{IVP}.

\end{lemma}
\begin{proof}
Uniqueness is classical (see \cite[Theorem 9.16]{daprato}).
Hence, we only need to check  that $f_{N, \al}$ is a solution to \eqref{IVP}.

A direct computation shows 
\[\partial_t \gamma_{\al + t}(x_n)= \frac{1}{4}\partial_{x_n}^2 \gamma_{\al+t}(x_n)\]

\noi
for $n = 1, \dots, N$.
Hence, 
it suffices to prove 
\[A(\bar x)\cdot \nabla \bigg( \prod_{n=1}^{2N}\gamma_{\al + t}(x_n)\bigg)=0.\]
Since 
\[\partial_{x_n}\gamma_{\al+t}(x_n)=-\frac{2x_n}{\al + t}\gamma_{\al + t}(x_n),\]

\noi
it suffices  to  check $A(\bar x) \cdot \bar x = 0$.
Recalling 
 $\bar x = (x_1, \dots, x_{2N}) = (p_1, \dots, p_N, q_1, \dots, q_N)$, 
 it follows from 
\eqref{pq3} and \eqref{pq56}
that 
\begin{align*}
A(\bar x) \cdot \bar x
& = 
 -\sum_{n = 1}^N n^3 q_n p_n + 
 \sum_{n = 1}^N \sum_{\substack{n = n_1 + n_2\\1\le |n_1|, |n_2| \le N}} n_2
(p_{n_1}q_{n_2} + q_{n_1} p_{n_2}) p_n  \\
& \quad +\sum_{n = 1}^N n^3 p_nq_n  - 
\sum_{n = 1}^N\sum_{\substack{n = n_1 + n_2\\1 \le |n_1|, |n_2| \le N}} n_2 
(p_{n_1}p_{n_2} - q_{n_1} q_{n_2})q_n\\
& = - \sum_{n= 1}^N \Re\big(\F_x ({\P_{N}(u_N \dx u_N)})(n)\big) \Re \ft u_N(n)\\
& \quad - \sum_{n= 1}^N \Im\big(\F_x ({\P_{N}(u_N \dx u_N)})(n)\big) \Im \ft u_N(n), 
\end{align*}

\noi
where $\F_x$ denotes the Fourier transform.
In the second step, we used the definition:
 $p_n(t) =\Re  \ft u_N(n, t)$
and $q_n(t) = \Im \ft u_N(n, t)$ 
together with 
\eqref{KdV10} and \eqref{pq2}.
By Parseval's identity with the fact that $u_N$ is real-valued and $u_N = \P_N u_N$, we then have 
\begin{align}
A(\bar x) \cdot \bar x
 = \frac 12 \int_\T \P_{N}(u_N \dx u_N) u_N dx
 = \frac 16 \int_\T \dx (u_N ^3)  dx = 0.
\label{ZX1}
\end{align}

\noi
We point out that, 
in view of  
\eqref{pq3}
and 
\eqref{ZX2}, 
$A(\bar x) \cdot \bar x = 0$ in \eqref{ZX1}
is equivalent to the conservation of $\l^2$-norm (= the Euclidean distance in $\R^{2N}$)
for the deterministic system:
\begin{align*}
 \dt p_n & = P_n,\\
\dt q_n & = Q_n
\end{align*}
\noi
for $n = 1, \dots, N$
(which in turn is equivalent to the conservation of the $L^2$-norm
for the  finite-dimensional KdV~\eqref{KdV7}).
This 
concludes the proof of Lemma \ref{LEM:stat}
\end{proof}

\begin{remark}\rm
A slight modification of the computations in the proof of Lemma \ref{LEM:stat} 
shows  the truncated white noise is invariant under 
the  finite-dimensional KdV dynamics~\eqref{KdV7}.
\end{remark}

As a corollary to Proposition \ref{PROP:finite}, 
we obtain the following tail estimate on 
 the size of solutions $u^N(t)$ to \eqref{KdV4}.

\begin{lemma}\label{LEM:tail}
Let $s <0$ and finite $p > 1$ such that $sp < -1$. 
Given $\al > 0$,  let $u^N$ be the solution to~\eqref{KdV4} with initial data 
$u_{0, \al}^\o = \sqrt \al\,  u_0^\o $ in \eqref{series5}.
  Then, we have 
\begin{equation}\label{tail}
\begin{split}
\PP\Big(\|u^N(t)\|_{\widehat{b}^s_{p,\infty}}>\lambda \bigg)
&=\PP\bigg( \sqrt{\frac{\al+t}{\al}}\|u_0^\o\|_{\widehat{b}^s_{p,\infty}}>\lambda\bigg)\\
&\le Ce^{-c\frac{\al}{\al + t} \lambda^2}.
\end{split}
\end{equation}

\noi
 for any $t \in \R_+$ and $\ld > 0$, 
 where the constants $C, c> 0$ are independent of  $\al > 0$.

\end{lemma}

The inequality in \eqref{tail} follows from 
the fact that $(\mu_1, \ft b^s_{p, \infty}(\T), L^2(\T))$ is an abstract Wiener space
when $sp < -1$ (\cite[Proposition 3.4]{OH1})
and  Fernique's theorem \cite{Fer}; see 
Theorem 3.1 in \cite{Kuo}.

\section{Review of the local well-posedness argument for SKdV}
\label{SEC:LWP}

In this section, we go over the local well-posedness
argument in \cite{OH4}
and collect useful estimates.

\subsection{Function spaces}
\label{SUBSEC:3a}

We first recall the definition of the $X^{s, b}$-spaces
adapted to the space $\ft b^s_{p, \infty}(\T)$ defined in \eqref{bs}.
Given $s \in \R$ and $1\le p, q \le \infty$, 
define 
the space $X^{s,b}_{p,q}(\T\times \R)$ by the norm:
\begin{align}
\|u\|_{X^{s,b}_{p,q}}=\|\langle n\rangle^s \langle \tau-n^3\rangle^b \widehat{u}(n,\tau)\|_{b^0_{p,\infty}L^q_\tau}.
\label{Xsb2}
\end{align}

\noi
In terms of the interaction representation $v(t) = S(-t) u(t)$, 
we have 
\[\|u\|_{X^{s,b}_{p,q}}=
\| \jb{\dx}^s \jb{\dt}^b v\|_{(\ft b^0_{p, \infty})_x \F L^{0, q}_t}, \]

\noi
where $\F L^{b, q}(\R)$ denotes the Fourier-Lebesgue space defined by the norm:
\begin{align}
\| f \|_{\F L^{b, q}} = \| \jb{\tau}^b \ft f(\tau)\|_{L^q_\tau}.
\label{FL1}
\end{align}

\noi
From \eqref{bs}, 
we have  $b^0_{p, \infty}(\Z) \supset \l^p(\Z) \supset \l^2(\Z)$ 
for $p \ge 2$, 
and thus we have
\begin{equation}
\|u\|_{X^{s ,b}_{p,2}} \le \|u\|_{X^{s,b}}
\label{embed0}
\end{equation}

\noi
for $p \ge 2$, 
where $X^{s, b}$ is the standard $X^{s, b}$-space defined in \eqref{Xsb}.
We also have 
\begin{align}
 \|u\|_{X^{-\frac 12 -\dl,b}}
\les
\|u\|_{X^{-\frac 12+ \dl ,b}_{p,2}}, 
\label{embed0a}
\end{align}

\noi
provided that $\dl > \frac{p-2}{4p}$ (with $p \ge 2$).
See \cite[eq.\,(17)]{OH4}.
See also the embedding \eqref{embed3} below.

Given an interval $I \subset \R_+$, 
we define the restriction space $X^{s,b}_{p,q}(I)$  of $X^{s,b}_{p,q}$ 
to the interval~$I$ 
by 
\begin{align}
\|u\|_{X^{s,b}_{p,q}(I)}=\inf\big\{\|v\|_{X^{s,b}_{p,q}(\T\times \R)}: v|_{I}=u\big\}.
\label{time1}
\end{align}

\noi
When $I = [0, T]$, we also set
$X^{s,b, T}_{p,q} =  X^{s,b}_{p,q}([0, T])$.
When $b > \frac 12$, it follows from  the Riemann-Lebesgue lemma that 
\begin{align}
X^{s, b}_{p, 2}(I) \subset  C(I; \ft b^s_{p,\infty}(\T))
 \label{embed1}
\end{align}

\noi
for any $s \in \R$ and $1 \le p \le \infty$.

When $q = 2$, in order to capture the temporal regularity of the stochastic convolution:
\begin{equation}
\Psi(t) = \int_0^t S(t-t')d W(t'), 
\label{psi1}
\end{equation}

\noi
where $W$ is as in \eqref{W1}, 
we need to take $b < \frac 12$ (see Lemma \ref{LEM:sto1} below), 
for which the embedding~\eqref{embed1} fails.
When $q = 1$, we have the following embedding:
\begin{align}
X^{s, 0}_{p, 1}(I) \subset  C(I; \ft b^s_{p,\infty}(\T)),
 \label{embed2}
\end{align}

\noi
and thus 
we use $X^{s, 0}_{p, 1}(I)$
as an auxiliary function space.

We now recall  the basic linear estimate for KdV;
given  $s\in\mathbb{R}$ and $0 \le b<\frac{1}{2}$, we have 
\begin{equation}
\|S(t)u_0\|_{X^{s,b,T}_{p,2}}\lesssim T^{\frac{1}{2}-b}\|u_0\|_{\widehat{b}^s_{p,\infty}}.
\label{lin1}
\end{equation}

\noi
for $0 < T \le 1$, 
where $X^{s,b,T}_{p,2} = X^{s,b}_{p,2}([0, T])$
is the restriction space defined in \eqref{time1}.
Next, we recall the $L^4$-Strichartz estimate due to 
 Bourgain \cite[Proposition 7.15]{BO1} (see also \cite[Proposition 6.4]{Tao2}):
\begin{equation}
\|u\|_{L^4(\mathbb{T}\times \mathbb{R})}\lesssim \|u\|_{X^{0,\frac{1}{3}}}.
\label{L4}
\end{equation}

\medskip

Lastly, we define the Fourier-Lebesgue space $\F L^{s. p}(\T)$ in the spatial variable by the norm:
\begin{align}
\| f \|_{\F L^{s, p}} = \| \jb{n}^s \ft f(n)\|_{\l^p_n}.
\label{FL2}
\end{align}

\noi
Then, for  $s \in \R$ and $1\le p, q \le \infty$, 
we define 
 the $X^{s, b}$-spaces
adapted to the Fourier-Lebesgue spaces by 
the norm:
\begin{align}
\|u\|_{Y^{s,b}_{p,q}}=\|\langle n\rangle^s \langle \tau-n^3\rangle^b \widehat{u}(n,\tau)\|_{\l^{p}_nL^q_\tau}.
\label{Xsb3}
\end{align}

\noi
Trivially, we have
\begin{align}
\| f\|_{\ft b^s_{p, \infty}} \le \| f \|_{\F L^{s, p}}
\qquad \text{and}\qquad 
\|u\|_{X^{s,b}_{p,q}} \le \|u\|_{Y^{s,b}_{p,q}}.
\label{FL3}
\end{align}

\noi
Given an interval $I \subset \R_+$, we define
the restriction space
 $Y^{s,b}_{p,q}(I)$ 
 as in \eqref{time1}.

\subsection{Partially iterated Duhamel formulation}
\label{SUBSEC:Duhamel}

In this subsection, 
we discuss the 
partially iterated Duhamel formulation
used in \cite{OH4}.

First, we consider the deterministic KdV \eqref{KDV}
considered in \cite{BO3}.
By writing it in the Duhamel formulation, we have
\begin{equation}
u(t)=S(t)u_0-\frac{1}{2}\N(u, u)(t), 
\label{DD1}
\end{equation}

\noi
where $\N(u_1, u_2)$ is given by 
\begin{align}
\N(u_1, u_2) (t) = \int_0^t S(t-t')\partial_x (u_1 u_2)(t')d t.
\label{DD2}
\end{align}

\noi
Note that the Fourier transform $\ft{u_1u_2}(n,\tau)$ can be written in the convolution form:
\begin{align*}
\ft{u_1u_2}(n,\tau) = 
\sum_{\substack{n = n_1+n_2}}\intt_{\tau = \tau_1+\tau_2}\ft{u_1}(n_1,\tau_1)\ft{u_2}(n_2,\tau_2)\,d\tau_1.
\end{align*}

\noi
Henceforth, we denote by $(n,\tau)$, $(n_1,\tau_1)$, 
and $(n_2,\tau_2)$ the space-time frequency  variables for the Fourier transforms
of $\N(u_1, u_2)$, $u_1$, and $u_2$ in \eqref{DD2}, respectively. 
In particular, we have
\begin{align}
 n = n_1 + n_2 \qquad \text{and}\qquad \tau = \tau_1 + \tau_2.
\label{DD4}
 \end{align}

\noi
By assuming that the initial data $u_0$ has spatial mean 0, 
it follows that $u(t)$ also has mean~0.
Furthermore, in view of the derivative on the nonlinearity, 
we  may assume that $n, n_1, n_2 \ne 0$.
We also denote the  modulations by 
\begin{align}
\s_0 = \jb{\tau - n^3} \qquad \text{and} \qquad \s_j = \jb{\tau_j - n_j^3}, 
\quad j = 1, 2.
\label{sig1}
\end{align}

\noi
Recall the following algebraic relation \cite{BO1}:
\begin{equation*} 
n^3 - n_1^3 - n_2^3 = 3 n n_1 n_2
\end{equation*}

\noi
for $n = n_1 + n_2$.
Then, under \eqref{DD4}, we have 
\begin{equation}
\MAX:= \max( \s_0, \s_1, \s_2) \gtrsim \jb{n n_1 n_2}.
\label{max}
\end{equation}

In our setting, we need to take $b < \frac 12$
to capture the temporal regularity of the stochastic convolution $\Psi$ in \eqref{psi1}.
On the other hand, 
the crucial bilinear estimate \eqref{bilin1}
holds only for $b = \frac 12$.
In order to overcome this difficulty, 
we decompose the nonlinearity into three pieces, 
depending on the sizes of the modulations $\s_0, \s_1$, and $\s_2$.
Define the sets  $M_j$, $j = 0 , 1, 2$,  by
\begin{align}
\begin{split}
M_0 &= \big\{(n, n_1, n_2, \tau, \tau_1, \tau_2) \in \mathbb{Z}_*^3 \times \R^3:
\s_0 = \MAX \big\}, \\
M_j &= \big\{(n, n_1, n_2, \tau, \tau_1, \tau_2) \in \mathbb{Z}_*^3 \times \R^3:
\s_j = \MAX \text{ and } \s_j > 1\big\}, \quad j = 1, 2.
\end{split}
\label{DD5}
\end{align}

\noi
For $j =  0, 1, 2$, 
let  $\N_j(u_1, u_2)$ be the contribution of  $\N(u_1, u_2)$ on $M_j$, 
and thus we have
\begin{align}
\begin{split}
\N(u_1,u_2)
& =\sum_{j=0}^2\N_j(u_1,u_2).
\end{split}
\label{NN2}
\end{align}

\noi
The standard bilinear estimate \eqref{bilin1} 
 allows us to estimate
 $\N_0(u_1,u_2)$ even when $b < \frac 12$;
 see \cite[eq.\,(46)]{OH4}.
 As for $\N_j(u_1, u_2)$, $j = 1, 2$,  however, 
 the bilinear estimate fails for temporal regularity $b < \frac12$
 (in $X^{s, b}_{p, q}$ for any $s \in \R$ and $1\le p, q\le \infty$)
 since, in this case, we do not have a sufficient power
for  the largest modulation $\s_j$
 to control the derivative loss in the nonlinearity.
 See \cite{KPV4}.

This issue was circumvented in \cite{BO3, OH4, OH6}
by partially iterating the Duhamel formulation~\eqref{DD1}
and writing it as 
 \begin{equation*}
u(t)=S(t)u_0-\frac{1}{2}\N_0(u, u)(t)
+ \frac 14 \N_1(\N(u, u), u)
+ \frac 14 \N_2(u, \N(u, u)).
\end{equation*}

\noi
Namely, for $j = 1, 2$, 
we replaced the $j$th entry in $\N_j(u, u)$
(where  the maximum modulation is given by $\s_j$)
by its Duhamel formulation \eqref{DD1}.
It follows from the definition of $\N_j$ (see~\eqref{DD5})
that there is no contribution from the linear solution $S(t)u_0$
in iterating the Duhamel formulation, 
since its space-time Fourier transform is supported on $\{\tau = n^3\}$, 
namely, $S(t) u_0$ has zero modulation, 
and thus from the definition of $\N_j$, we have  
$ \N_1(S(t)u_0,  u)
= \N_2(u, S(t)u_0) = 0$.

 In the context of SKdV \eqref{KdV1}
and its Duhamel formulation \eqref{KdV3}:
\begin{equation}
u(t)=S(t)u_0-\frac{1}{2}\N(u, u) +\Psi, 
\label{DD7a}
\end{equation}

\noi
 the discussion above leads to 
 \begin{align}
\begin{split}
u(t)& = S(t)u_0-\frac{1}{2}\N_0(u, u)(t)\\
& \quad + \frac 14 \N_1(\N(u, u), u)
- \frac 12 \N_1(\Psi, u)\\
& \quad + \frac 14 \N_2(u, \N(u, u))
- \frac 12 \N_2(u, \Psi) + \Psi, 
\end{split}
\label{DD7}
\end{align}

\noi
where $\Psi$ is the stochastic convolution in \eqref{psi1}.
In \cite{OH4}, the first author studied this new formulation \eqref{DD7}
and establish an a priori bound on solutions
(with smooth initial data and (spatially) smooth noise), 
which allowed him to construct a solution \eqref{KdV3} by an approximation argument.

Lastly, we state an analogous formulation for  the truncated SKdV \eqref{KdV4}.
By writing~\eqref{KdV4} in the Duhamel formulation, we have
\begin{equation}
u^N(t)=S(t)u_0-\frac{1}{2}\N^N(u^N, u^N) +\Psi, 
\label{DD8}
\end{equation}

\noi
where $\N^N(u^N, u^N)$ is given by 
\begin{align*}
\N^N(u_1, u_2)
& = \P_N \N(\P_N u_1, \P_N u_2)\\
& =  \int_0^t S(t-t')\partial_x \P_N (\P_N u_1 \P_N u_2)(t')d t.
\end{align*}

\noi
Then, by partially iterating the Duhamel formulation as above, 
we rewrite \eqref{DD8} as 
 \begin{align}
 \begin{split}
u^N(t)& = S(t)u_0-\frac{1}{2}\N^N_0(u^N, u^N)(t)\\
& \quad + \frac 14 \N^N_1(\N^N(u^N, u^N), u^N)
- \frac 12 \N_1^N(\Psi, u^N)\\
& \quad + \frac 14 \N_2^N(u^N, \N^N(u^N, u^N))
- \frac 12 \N_2^N(u^N, \Psi) + \Psi
\end{split}
\label{DD9}
\end{align}

\noi 
where $\N_j^N(u_1, u_2)$
is the contribution of  $\N^N(u_1, u_2)$ on $M_j$, $j = 0, 1, 2$.

\subsection{Local well-posedness
and an a priori bound}
\label{SUBSEC:3b}

In this subsection, we collect 
the useful nonlinear estimates
on the iterated formulation \eqref{DD7}
from \cite{OH4}, 
and establish an a priori bound
for solutions to the truncated SKdV \eqref{KdV4}.

We first recall the following local well-posedness result of SKdV \eqref{KdV1} from~\cite{OH4}.

\begin{theorem}\label{THM:hiro}
  Let $s=-\frac{1}{2}+\delta$ 
and   $p = 2 + \dl_0$ 
for some small $\dl, \dl_0 > 0$ such that 
 $\frac{p-2}{4p}<\delta<\frac{p-2}{2p}$. 
Given a mean-zero function $u_0\in \widehat{b}^s_{p,\infty}(\mathbb{T})$, 
there  exist a stopping time $T_\omega>0$ and a unique solution 
$u\in C([0,T_\omega];\widehat{b}^s_{p,\infty}(\mathbb{T})) \cap X^{s, \frac 12 - \dl}_{p, 2}([0, T_\o])$ 
to \eqref{KdV1} with $u|_{t = 0} = u_0$.
Furthermore, denoting by $T_* = T_*(\o)$ 
the maximal time of existence, 
we have the following blowup alternative\textup{:}
\begin{align*}
\lim_{t \nearrow T_*}\| u(t) \|_{\ft b^s_{p, \infty}} = \infty
\qquad 
\text{or}\qquad  T_* = \infty.
\end{align*}

\noi

\end{theorem}

Here, the condition $\delta<\frac{p-2}{2p}$ is equivalent to   $sp<-1$, 
while $\dl > \frac{p-2}{4p}$ is used for the embedding \eqref{embed0a}.

In the following, we recall some of the nonlinear estimates from \cite{OH4}.
In the remaining part of the paper, 
we fix small $\dl, \dl_0>0$, satisfying the hypothesis in Theorem \ref{THM:hiro}, 
and set 
\begin{align}
 \al = \frac 12 - \dl
 \label{al}
\end{align} 

\noi
as in \cite{BO3, OH4}.
The following discussion applies to both the original SKdV \eqref{KdV1}
and its truncated version \eqref{KdV4}.
In order to treat them in a uniform manner, 
we take 
$N \in \Z_{\ge 0}$
and 
set 
$u^\infty = u$, $\N^\infty(u_1, u_2) = \N(u_1, u_2)$, 
and $\N^\infty_j(u_1, u_2) = \N_j(u_1, u_2)$, $j = 0, 1, 2$.
Note that, in the following,  all the estimates hold, 
uniformly in $N \in\Z_{\ge 0}$.

Given $T > 0$, 
we define the random quantity $L_\omega(T)$  by
\begin{align}
  L_\omega(T)&= 
  \|\ind_{[0, T]}\Psi\|_{X^{-\frac 12 - \frac 12 \dl, \frac 12 -\dl}}
+ 
\|\ind_{[0, T]}\Psi\|_{Y^{-\frac 12 - \frac 12 \dl, \frac {11}{16} +\dl}_{2, 4}}, 
\label{F1}
\end{align}

\noi
which is a pathwise\footnote{In \cite[``Estimate on (ii)'' on pp.\,296-297]{OH4}, 
an expectation was taken on 
 the $X^{-\al, 1-\al, T}$-norms of 
 $\N_1(\Psi, u)$.  However, we in fact need a pathwise bound, which 
is  established in Appendix \ref{SEC:B}.} upper bound 
for the $X^{-\al, 1-\al, T}$-norms of 
 $\N_1(\Psi, u)$ in~\eqref{DD7}
and $\N_1^N(\Psi, u^N) $ in~\eqref{DD9}.
See Appendix \ref{SEC:B}.
We point out that, while the analysis in Appendix~\ref{SEC:B}
(see \eqref{BX2} and \eqref{BX3})
yields the spatial regularity $- \frac 12 - \dl$, 
we use a slightly worse spatial regularity for the definition of  $L_\o(T)$
in \eqref{F1}
(so that the estimate \eqref{F4} below holds, allowing us to gain a decay in $N$).
Note that, in contrast to~\cite{OH4}, we defined~$L_\omega(T)$ on the  ``long'' interval $[0,T]$. 
This will be useful in Sections \ref{SEC:prob} and \ref{SEC:5}
 when we iterate the local-in-time argument
 on many small subintervals of $[0,T]$ but with a fixed driving space-time white noise.
From \eqref{F1} and 
Remark \ref{REM:psi1} below, 
we have
\begin{align}
\| L_\o(T)\|_{L^r(\O)} \les \sqrt r\, T^{\frac 32}
\label{O0}
\end{align}

\noi
for any $T> 0$ and $1 \le r < \infty$, 
provided that $\dl > 0$ is sufficiently small 
such that 
\[\bigg(\frac {11}{16} +\dl - 1\bigg) 4 < -1.\]

With this notation, 
the main nonlinear estimate \cite[eq.\,(73)]{OH4} 
(see also  Appendix~\ref{SEC:B}) reads
(with some small $\ta > 0$)
\begin{equation}
\begin{split}
\|u^N\|_{X^{-\alpha,\alpha,T_1}_{p,2}}
& \le
C_1\|u_0^N\|_{\widehat{b}^{-\alpha}_{p,\infty}}
+\frac{1}{2}C_2T_1^{\theta}\|u^N\|_{X^{-\alpha,\alpha,T_1}_{p,2}}^2
+2C_3T_1^{\theta}\|u^N\|^3_{X^{-\alpha,\alpha,T_1}_{p,2}}\\
&\quad +2C_3T_1^{\theta}L_\omega(T) \|u^N\|_{X^{-\alpha,\alpha,T_1}_{p,2}}+C_4\| \Psi\|_{X_{p,2}^{-\alpha,\alpha,T_1}}
  \end{split}
\label{nonlin1}
\end{equation}

\noi
for  any  $T> 0$ and $0 <  T_1 \le \min(1, T)$, 
provided that 
\begin{align}
C_3 T_1^\ta R \le \frac 12
\qquad \text{and} 
\qquad \|u^N\|_{X^{-\alpha,\alpha,T_1}_{p,2}} \le R.
\label{nonlin1a}
\end{align}

\noi
Here, in importing \eqref{nonlin1}
from \cite{OH4}, we use the fact that $\P_N$ is bounded on relevant function spaces, uniformly in $N \in \NB$.
The nonlinear estimate \eqref{nonlin1}
together with an analogous difference estimate (\cite[eq.\,(74)]{OH4})
allows us to construct a solution $u \in 
X^{-\alpha,\alpha,T_1}_{p,2}$ to \eqref{KdV1}
as a limit of smooth solutions.

Next, we  estimate the $\ft b^{-\al}_{p, \infty}$-norm  of
a solution $u^N(t)$ to \eqref{DD8}.
Since $\al < \frac 12$, the embedding \eqref{embed1} does not hold
and thus \eqref{nonlin1} is not directly applicable.
However, some  terms can be estimated in a stronger norm.
Indeed, from \eqref{embed1}, \eqref{embed0},  and \cite[eq.\,(47) and (72)]{OH4}, we have, 
under the condition \eqref{nonlin1a},  
\begin{align}
\begin{split}
\| \N_1^N(& u^N, u^N) + \N_2^N(u^N, u^N)\|_{C([0, T_1]; \ft b^{-\al}_{p, \infty})}\\
& \les 
\| \N_1^N( u^N, u^N) + \N_2^N(u^N, u^N)\|_{X^{-\alpha,1-\al,T_1}}\\
& \le 
2C_3\Big(T_1^{\theta}\|u^N\|^3_{X^{-\alpha,\alpha,T_1}_{p,2}}
 + T_1^{\theta}L_\omega(T) \|u^N\|_{X^{-\alpha,\alpha,T_1}_{p,2}}
 \Big)
 \end{split}
\label{nonlin2}
\end{align}

\noi
for  any  $T> 0$ and $0 <  T_1 \le \min(1, T)$, 
where the first inequality follows from \eqref{embed1} since 
 $b = 1-\al = \frac 12 + \dl> \frac 12$.
As for $\N_0^N$, we write it as 
\begin{align*}
\N_0^N( u^N, u^N) = \N_3^N( u^N, u^N)+ \N_4^N( u^N, u^N),
\end{align*}

\noi
where $\N_3^N$ denotes the contribution of $\N_0^N$ on $\{\max(\s_1, \s_2) \ges \jb{n n_1n_2}^\frac{1}{100} \}$.
Then, from 
\eqref{embed1}, 
\eqref{embed2}, 
and 
\cite[(a) and (b) on p.\,302)]{OH4}, we have 
\begin{align}
\begin{split}
\| \N_0^N(u^N, u^N)\|_{C([0, T_1]; \ft b^{-\al}_{p, \infty})}
& \les \| \N_3^N(u^N, u^N)\|_{X_{p, 2}^{-\alpha,1-\al,T_1}}
+ \| \N_4^N(u^N, u^N)\|_{X_{p, 1}^{-\alpha,0,T_1}}\\
& \les 
T_1^\ta
\|u^N \|^2_{X^{-\alpha,\alpha,T_1}_{p,2}}.
\end{split}
\label{nonlin3}
\end{align}

\noi
Hence, putting \eqref{DD8}, \eqref{nonlin2}, and \eqref{nonlin3}
together, we obtain
\begin{align}
\begin{split}
\|u^N\|_{C([0, T_1]; \ft{b}^{-\al}_{p,\infty})}
&\le \|u_0\|_{\widehat{b}^{-\alpha}_{p,\infty}} 
+ C_5T_1^{\theta}\|u^N\|^2_{X_{p,2}^{-\alpha,\alpha,T_1}} \\
&\quad +C_6 T_1^{\theta}\|u^N\|^3_{X_{p,2}^{-\alpha,\alpha,T_1}}
+C_7 T_1^{\theta}L_\o(T)\|u^N\|_{X_{p,2}^{-\alpha,\alpha,T_1}}\\
& \quad +\|\Psi\|_{C([0, T_1]; \ft{b}^{-\al}_{p,\infty})}
\end{split}
\label{nonlin4}
\end{align}

\noi
for  any  $T> 0$ and $0 <  T_1 \le \min(1, T)$, 
provided that  \eqref{nonlin1a} holds.

\medskip

We now state an a priori bound on a solution $u^N$
to the truncated SKdV \eqref{KdV4}.

\begin{lemma}\label{LEM:nonlin2}

Let $N \in \NB$.
Then, there exist absolute  constants $\g > 0$ 
and $C_*, c_*  > 0 $ such that, 
given any $T> 0$, we have
\begin{equation}\label{bd1}
\|u^N\|_{X^{-\alpha,\alpha}_{p,2}(I)}
\le C_* \Big(\|u^N(t_0)\|_{\widehat{b}^{-\alpha}_{p,\infty}}
+\| \Psi\|_{X^{-\alpha,\alpha}_{p,2}([t_0, t_0+1])}  + 1\Big)
=: R_\o(t_0)
\end{equation}

\noi
for any time interval $I = [t_0, t_0 + T_1] \subset [0, T]$ of length $T_1 \le 1$
and
any solution  $u^N$ to the truncated SKdV \eqref{KdV4}, 
provided that 
 \begin{align}
  T_1\le c_* (R_\o(R_\o+1)+L_\o(T))^{-\gamma}.
\label{bd1a}
\end{align}

\noi
Here, the constants 
$\g > 0$ 
and $C_*, c_*  > 0 $ are independent of $N \in \NB$.

\end{lemma}

\begin{proof}

Under \eqref{nonlin1a}, 
it follows from \eqref{nonlin1} that 
\begin{equation}
\|u^N\|_{X^{-\alpha,\alpha}_{p,2}(I)}
 \le
2C_1\|u^N(t_0)\|_{\widehat{b}^{-\alpha}_{p,\infty}}
+2C_4\| \Psi\|_{X_{p,2}^{-\alpha,\alpha}(I)}, 
\label{bd2}
\end{equation}

\noi
provided that 
\begin{equation*}
T_1^\ta\bigg(  \frac{1}{2}C_2R+2C_3R^2+2C_3L_\omega(T) \bigg) \le \frac{1}{2}.
\end{equation*}

\noi
Then, under \eqref{bd1a} with $\g = \ta^{-1}$, 
the bound \eqref{bd1} follows from 
\eqref{bd2} and 
 a continuity argument.
\end{proof}

\begin{remark}\label{REM:cont}\rm
In order use a continuity argument
in the proof of Lemma \ref{LEM:nonlin2}
presented above, we need the continuity 
of the $X^{-\al, \al}_{p, 2}([t_0, t_1])$-norm with respect to the right endpoint~$t_1$.
While it may be possible to check this directly
(see, for example, \cite[Appendix A]{BOP2}
and \cite[Lemma 8.1]{GO}), 
let us use the following equivalence:
\begin{align}
\| u \|_{X^{-\al, \al}_{p, 2}([t_0, t_1])} \sim
\| \ind_{[t_0, t_1]}u \|_{X^{-\al, \al}_{p, 2}}, 
\label{equiv}
\end{align}

\noi
where the norm on the left-hand side is defined in \eqref{time1}, 
and study the latter norm.
Recall that 
the equivalence \eqref{equiv}
holds
since the temporal regularity $b = \al = \frac 12 - \dl$ is below $ \frac 12$
(see \cite[eq.\,(3.5)]{CO}).
Such equivalence also holds for the general 
$X^{s, b}_{p, q}([t_0, t_1])$
for $0 \le b < \frac{q-1}{q}$; see \cite{COZ}.

Given small $h > 0$, from the triangle inequality, we have 
\begin{align}
\| \ind_{[t_0, t_1 + h ]}u \|_{X^{-\al, \al}_{p, 2}}
- \| \ind_{[t_0, t_1]}u \|_{X^{-\al, \al}_{p, 2}}
\le
\| \ind_{[t_1, t_1 + h ]}u \|_{X^{-\al, \al}_{p, 2}},
\label{equiv2}
\end{align}

\noi
and thus it suffices to show that the right-hand side of \eqref{equiv2}
tends to $0$ as $h \to 0$.
In view of the definition
\eqref{Xsb2}, such a claim follows once we prove 
\begin{align}
\lim_{h \to 0}\| \ind_{[t_1, t_1 + h ]}f \|_{H^\al} = 0
\label{equiv3}
\end{align}

\noi
for a function $f \in H^\al(\R)$.
Obviously, we have $\lim_{h \to 0}\| \ind_{[t_1, t_1 + h ]}f \|_{L^2} = 0$.
Using the physical side characterization of the homogeneous Sobolev norm, 
we have
\begin{align*}
\| \ind_{[t_1, t_1 + h ]}f \|_{H^\al}^2
& = 
\int_{\R}\int_{\R} \frac{|\ind_{[t_1, t_1+h]}(t)f(t) - \ind_{[t_1, t_1+h]}(\tau)f(\tau)|^2}{|t - \tau|^{1 + 2\al}}dt d\tau\\
& =  \1(h) + \II(h) + \III(h),
\end{align*}

\noi
where $\1$, $\II$, and $\III$ are defined by 
\begin{align*}
\1(h) & = \int_{[t_1, t_1 + h]}
\int_{[t_1, t_1 + h]}
 \frac{|f(t) - f(\tau)|^2}{|t - \tau|^{1 + 2\al}}dt d\tau, \\
\II(h)& = \int_{[t_1, t_1 + h]}
\int_{[t_1, t_1 + h]^c}
 \frac{|f(t)|^2}{|t - \tau|^{1 + 2\al}}d\tau dt, \\
 \III(h)& = \int_{[t_1, t_1 + h]}
\int_{[t_1, t_1 + h]^c}
 \frac{|f(\tau)|^2}{|t - \tau|^{1 + 2\al}}dt d\tau .
\end{align*}

\noi
By the dominated convergence theorem
with the fact that $f \in H^\al(\R)$, 
we see that $\lim_{h \to \infty} \1(h) = 0$.
As for $\II(h)$, integration in $\tau$ yields
\begin{align*}
\II(h)& \sim  \int_{[t_1, t_1 - h]}
 \frac{|f(t)|^2}{|t - t_1 -h|^{ 2\al}} dt
 + 
 \int_{[t_1, t_1 + h]}
 \frac{|f(t)|^2}{|t - t_1 |^{ 2\al}} dt\\
 & \les \|f \|_{\dot H^\al(\R)}, 
\end{align*}

\noi
where the second step follows from 
Hardy's  inequality (\cite[Lemma A.2]{Tao}) 
since $0 \le \al < \frac 12$.
Noting that $\III(h) = \II(h)$, we see that the term $\III(h)$ 
also satisfies the bound above.
Also, the case $h < 0$ follows from an analogous consideration.
Putting everything together, 
we conclude \eqref{equiv3}.
See also Lemma 4.4 in \cite{Bring}.

\end{remark}

   \medskip

We conclude this section by 
stating a  lemma on growth
of the stochastic convolution~$\Psi$ in~\eqref{psi1}
over long time intervals.
We point out that analogous regularity results
were obtained in \cite[Propositions 4.1 and 4.5]{OH4}
but they are only for short times.

\begin{lemma}\label{LEM:sto1}

Let $s < 0$ and $1 \le p, q < \infty$ such that $s p < -1$.

\smallskip

\noi
\textup{(i)}
Let $(b-1) q < -1$. Given any $1 \le r < \infty$ and $T \ge1 $, we have 
\begin{align}
\Big\| \| \Psi \|_{X^{s, b}_{p, q}(I)}\Big\|_{L^r (\O)} \le 
\Big\| \| \Psi \|_{Y^{s, b}_{p, q}(I)}\Big\|_{L^r (\O)} \les 
\sqrt {r  T}
\label{sto1}
\end{align}

\noi
for any interval $I \subset [0, T] $ with length $|I| \le 1$, 
where the implicit constant is independent of $r$ and $T$.

\medskip

\noi
\textup{(ii)}
Given any $1 \le r  < \infty$ and $T \ge1 $, we have 
\begin{align}
\Big\| \| \Psi \|_{C([0, T]; \ft b^s_{p, \infty})}\Big\|_{L^r (\O)} \les 
\sqrt {r  T \log T}, 
\label{sto2}
\end{align}

\noi
where the implicit constant is independent of $r $ and $T$.

\end{lemma}

We present the proof of Lemma \ref{LEM:sto1} in Appendix \ref{SEC:A}.

\begin{remark}\label{REM:psi1}\rm
We point out that the bound \eqref{sto1} 
holds
only for intervals $I$ of short lengths.
Indeed, 
a slight modification of the proof
yields the following estimate for $I = [0, T]$:
\begin{align}
\Big\| \| \Psi \|_{X^{s, b, T}_{p, q}}\Big\|_{L^r(\O)} \le 
\Big\| \| \Psi \|_{Y^{s, b, T}_{p, q}}\Big\|_{L^r(\O)} \les 
\sqrt r\, T^{\frac 32},
\label{sto3}
\end{align}

\noi
 where the right-hand side is much worse than those in \eqref{sto1}
and \eqref{sto2}.
See Remark \ref{REM:psi2}.

\end{remark}

\section{Probabilistic uniform growth bound}
\label{SEC:prob}

Given $N \in \NB$, let $u^N$ be the global solution to the truncated SKdV \eqref{KdV4}
with the mean-zero white noise initial data $u_0^\o$ in \eqref{series4}.
Our main goal in this section is to establish the following probabilistic growth bound
on the solution $u^N$ to \eqref{KdV1}
whose proof is based on a variant of Bourgain's invariant measure argument
in the current setting of an evolution system of measures
(Proposition \ref{PROP:finite}).

\begin{proposition}\label{PROP:main}

Let $\al$ and $p$ be as in \eqref{al} and Theorem \ref{THM:hiro}, respectively.
Given any  $T\gg1$ and $0 < \eps\ll 1$, there exists a set  $\Omega_{T,\eps}(N)$ such that 
$\mathbb{P}(\Omega_{T,\eps}(N)^c)<\eps$ and
  \begin{align}
    \sup_{t\in [0,T]}\|u^N(t)\|_{\widehat{b}_{p,\infty}^{-\alpha}}&\le  C
 \sqrt {\log \frac 1 \eps}\sqrt{T \log T}
 \label{pro1}
  \end{align}

\noi
on $\Omega_{T,\eps}(N)$, 
where the constant $C > 0$ is independent of $N \in \NB$, $T\gg1$, and $\eps \ll 1$.

\end{proposition}

  \begin{proof}
  
Fix small $T_1 > 0$ (to be chosen later), 
and let $I_j = [j T_1, (j+1) T_1] \cap [0, T]$, $j \in \Z_{\ge 0}$.
Recall from 
Proposition \ref{PROP:finite} that 
 the solution $u^N(jT_1)$ at time $t = j T_1$
 is distributed by 
 the white noise measure $\mu_{1 + jT_1}$
 with variance $1 + j T_1$, 
 where $\mu_{1 + jT_1}$ is as in \eqref{mu7}.
 Then, given $K_1 \gg 1$, set $\O_1 = \O_1(T, \eps, N) \subset \O$ by 
 \begin{align}
 \Omega_1= \bigcap_{j=0}^{[T/T_1]}
 \Big\{\|u^N(j T_1 )\|_{\widehat{b}^{-\alpha}_{p,\infty}}\le K_1\Big\}.
\label{O1}
 \end{align}

\noi
Then,  it follows from Lemma \ref{LEM:tail}
and choosing 
\begin{align}
K_1=r_1\sqrt{T \log \frac T\eps}
\label{O1a}
\end{align}

\noi
for some $r_1 \gg1 $(to be chosen later)
that 
\begin{align}
\begin{split}
\PP(\Omega_1^c)&\les \sum_{j=0}^{[T/T_1]}
e^{-\frac{c}{1+j T_1}K_1^2} 
\sim \frac T{T_1} e^{-\frac{c'}{T}K_1^2} 
= T_1^{-1} T^{1-c'r_1^2}\eps^{c'r_1^2}.
\end{split}
\label{O2}
\end{align}

\noi
Next, define $\O_2 = \O_2(T, \eps) \subset \O$ by 
 \begin{align}
 \Omega_2= \bigcap_{j=0}^{[T/T_1]}
 \Big\{
 \| \Psi\|_{X^{-\alpha,\alpha}_{p,2}([jT_1, jT_1 + 1])}
\le K_1\Big\}, 
\label{O3}
 \end{align}

\noi
where $K_1$ is as in \eqref{O1a}.
Then, by
Lemma \ref{LEM:sto1}\,(i) and Chebyshev's inequality, 
we have 
\begin{align}
\begin{split}
\PP(\Omega_2^c)&\les \sum_{j=0}^{[T/T_1]}
e^{-\frac{c}{1+j T_1}K_1^2} 
\sim 
T_1^{-1} T^{1-c'r_1^2}\eps^{c'r_1^2}
\end{split}
\label{O4}
\end{align}

\noi
just as in \eqref{O2}.
Lastly, define $\O_3 = \O_3(T, \eps) \subset \O$ by 
\begin{align}
\O_3 = \big\{ L_\o(T) \le K_2\big\}, 
\label{O5}
\end{align}

\noi
where $L_\o(T)$ is as in \eqref{F1} and 
\begin{align}
K_2 = r_2 \sqrt {T^3 \log \frac 1\eps} .
\label{O5a}
\end{align}
Then, by choosing $r_2 > 0$ sufficiently large, 
it follows from \eqref{O0} 
and Chebyshev's inequality
that 
\begin{align}
\PP(\O_3^c) \leq C e^{- \frac c {T^3} K_2^2} < \frac\eps 4.
\label{O6}
\end{align}

Let $R_\o$ be as in \eqref{bd1}.
Then, on $\O_1\cap \O_2 \cap \O_3$, we have
\begin{align}
R_\o(jT_1) \leq C_*(2K_1 + 1) \sim K_1
\qquad \text{and}\qquad
L_\o(T) \le K_2
\label{O7}
\end{align}

\noi
for $j = 0, 1, \dots, \big[\frac T {T_1}\big]$.
In view of \eqref{nonlin1a} and \eqref{bd1a} in 
Lemma \ref{LEM:nonlin2} with \eqref{O7}, we now choose $T_1 > 0$ by 
setting
\begin{align}
T_1 
\sim \min \Big\{ K_1^{-\frac 1\ta}, \,(K_1^2 + K_2)^{-\g}\Big\}.
\label{O8}
\end{align}

\noi
Then, by choosing $r_1 > 0$ sufficiently large, 
it follows from \eqref{O2} and \eqref{O4}
with 
\eqref{O1a},  \eqref{O5a}, and \eqref{O8}
that 
\begin{align}
\PP(\Omega_k^c)
< \frac \eps 4
\label{O9}
\end{align}

\noi
for $k = 1,2$.
Furthermore, from Lemma \ref{LEM:nonlin2}
and \eqref{O7} with \eqref{O1a}, 
we obtain
\begin{equation}\label{O9a}
\|u^N\|_{X^{-\alpha,\alpha}_{p,2}(I_j)}
\le C_*(2K_1 + 1)\sim  K_1
\sim \sqrt{T \log \frac T\eps}
\end{equation}

\noi
for $j = 0, 1, \dots, \big[\frac T {T_1}\big]$.

Now, define $\O_4 = \O_4(T, \eps) \subset \O$ by 
\begin{align}
\O_4 = \bigg\{\| \Psi \|_{C([0, T]; \ft b^s_{p, \infty})} \le r_3 
\sqrt {\log \frac 1 \eps}\sqrt{T \log T}\bigg\}.
\label{O10}
\end{align}

\noi
Then, from 
Lemma \ref{LEM:sto1}\,(ii) and Chebyshev's inequality, we have
\begin{align}
\PP(\O_4^c) < \frac \eps 4
\label{O11}
\end{align}

\noi
by choosing $r_3 > 0$ sufficiently large.
In view of \eqref{nonlin4} with \eqref{O7} and \eqref{O9a}, 
we further impose that 
\begin{align}
T_1^\ta\big(C_5  C_*(2K_1 + 1) 
+ C_6 C_*^2(2K_1 + 1)^2 + C_7K_2\big) \leq 1 .
\label{O12}
\end{align}

\noi
Note that \eqref{O12} yields
$T_1 \les (K_1^2 + K_2)^{-\frac 1\ta}$, 
which is essentially implied by \eqref{O8}
(by possibly making $r_1$ larger)
and thus the bound \eqref{O9} still holds.

Finally, set $\Omega_{T,\eps}(N) = \O_1\cap \cdots \cap \O_4$.
Then, from \eqref{O6}, \eqref{O9}, and \eqref{O11}, 
we have
\[ \PP(\Omega_{T,\eps}(N)^c) < \eps.\]

\noi
Furthermore, 
on $\Omega_{T,\eps}(N)$, 
we conclude from 
\eqref{nonlin4} with 
\eqref{O1}, \eqref{O1a},  \eqref{O9a},
\eqref{O10}, and  \eqref{O12} that 
\begin{align*}
\|u^N\|_{C(I_j ; \ft{b}^{-\al}_{p,\infty})}
\les \sqrt {\log \frac 1 \eps}\sqrt{T \log T}, 
\end{align*}

\noi
uniformly in $j = 0, 1, \dots, \big[\frac T {T_1}\big]$,
which implies \eqref{pro1}.
\end{proof}

\section{Approximation argument}
\label{SEC:5}

In this section, we present the proof of 
Theorem \ref{THM:1}.
We first establish the following 
`almost' almost sure global well-posedness
of SKdV \eqref{KdV1}
via an approximation argument.

Given $N \in \NB$, let $u^N$ be the global solution to the truncated SKdV \eqref{KdV4}
with the mean-zero white noise initial data $u_0^\o$ in \eqref{series4},
and 
let $u$ be the  solution to  SKdV \eqref{KdV1}
with the mean-zero white noise initial data $u_0^\o$ in \eqref{series4},
whose local existence is guaranteed by Theorem \ref{THM:hiro}.

\begin{proposition}\label{PROP:GWP1}
  Let $\al =\frac{1}{2}- \delta$ 
and   $p = 2 + \dl_0$ 
for some small $\dl, \dl_0 > 0$ such that 
 $\frac{p-2}{3p}<\delta<\frac{p-2}{2p}$. 
Given any  $T\gg 1$ and $0 < \eps \ll 1$, 
there exist a set  $\Omega_{T,\eps}$
and $N_* = N_*(T, \eps) \in \NB$ such that 
$\mathbb{P}(\Omega_{T,\eps}^c)<\eps$ and, 
on $\O_{T, \eps}$, we have 
  \begin{align}
\sup_{t\in [0,T]}\| u(t) - u^{N_*}(t)\|_{\widehat{b}_{p,\infty}^{-\alpha}}
\le C(T, \eps) N_*^{-\frac{\dl}{2}}.
\label{XX2}
\end{align}

\noi
In particular, 
on $\O_{T, \eps}$, 
the solution $u$ to SKdV \eqref{KdV1}
with the mean-zero white noise initial data $u_0^\o$ in \eqref{series4}
exists on the  time interval $[0, T]$.

\end{proposition}

As compared to Theorem \ref{THM:hiro}, 
we need an extra restriction $ \dl > \frac{p-2}{3p}$  in order to obtain a decay in $N$.
See \eqref{embed3}
and \eqref{X5a}.

\begin{proof}

We first record the following embedding, which requires
the additional condition
$\dl > \frac {p-2}{3p}$.
Let $p > 2$.
By H\"older's inequality, we have 
\begin{align*}
\| f\|_{H^{-\frac 12 - \frac 12\dl}} 
& \le 
\bigg(\sum_{j = 0}^\infty 2^{-2\eps j}
 \| \jb{n}^{-\frac{1}{2}- \frac 12 \dl + \eps} \ft f(n)\|^2_{\l^2_{|n|\sim 2^j}}\bigg)^\frac 12 
\\
& \le 
\| \jb{n}^{-\frac 32 \dl + \eps} \|_{\l^{\frac {2p}{p - 2}}_n}
\sup_{j \in \Z_{\ge 0}} \| \jb{n}^{-\frac 12 + \dl} \ft f(n) \|_{\l^p_{|n|\sim 2^j}}
 \les \| f\|_{\ft b^{-\frac 12 + \dl}_{p, \infty}}, 
\end{align*}

\noi
provided that $\dl > \frac {p-2}{3p}$ (by taking $\eps > 0$ sufficiently small).
Hence, we have 
\begin{align}
\| u \|_{X^{ -\frac 12 - \frac 12\dl, b}}
\les \| u \|_{X^{ - \frac 12 + \dl, b}_{p, 2}}
\label{embed3}
\end{align}

\noi
for any $s, b \in \R$, 
provided that $\dl > \frac {p-2}{3p}$.
Instead of  \eqref{embed0a}, 
 we use \eqref{embed3} in the following.

\medskip

\noi
$\bullet$ {\bf Step 1:}
In the following, we  first study the difference
of the Duhamel formulations 
\eqref{DD7a}
and \eqref{DD8}
for SKdV \eqref{KdV1} and the truncated SKdV \eqref{KdV4}, 
respectively, 
on {\it short} time intervals.
Our first main goal is to estimate the difference 
\begin{align*}
\| \N(u, u) - \N^N(u^N, u^N)\|_{X_{p, 2}^{-\al, \al, T_1}}
\end{align*}

\noi
for small $T_1 > 0$, 
where $\al = \frac 12 - \dl$ as in 
\eqref{al}.
From the discussion in Subsection \ref{SUBSEC:Duhamel}, 
we have
\begin{align*}
\N(u, u) - \N^N(u^N, u^N)
& = \sum_{j = 0}^2 \Big(\N_j(u, u) - \N^N_j(u^N, u^N)\Big)\\
& = \N_0(u, u) - \N^N_0(u^N, u^N)\\
& \quad 
- \frac 12\Big(
\N_1(\N(u, u), u) - \N_1^N(\N^N(u^N, u^N), u^N)\Big)\\
& \quad + 
 \N_1( \Psi, u) - \N_1^N(\Psi, u^N)\\
& \quad 
- \frac 12\Big(
\N_2(u, \N(u, u)) - \N_2^N(u^N, \N^N(u^N, u^N))\Big)\\
& \quad + 
\N_2(u,  \Psi) - \N_2^N(u^N, \Psi).
\end{align*}

\noi
From the definitions of $\N_1(u, u)$ and $\N_1^N(u^N, u^N)$, we have 
\begin{align}
\begin{split}
& \N_1(\N(u, u), u) - \N_1^N(\N^N(u^N, u^N), u^N)\\
& \quad = \N_1(\N(u, u), u) - \P_N \N_1(\P_N\N(\P_Nu^N, \P_Nu^N), \P_Nu^N)\\
& \quad = \N_1\big(\N(u, u) - \P_N\N(\P_Nu^N, \P_Nu^N), u\big) \\
& \quad \quad  + \N_1(\P_N\N(\P_Nu^N, \P_Nu^N), u - u^N)\\
& \quad \quad  + \N_1(\P_N\N(\P_Nu^N, \P_Nu^N), \P_N^\perp u^N)\\
& \quad \quad  +\P_N^\perp \N_1(\P_N\N(\P_Nu^N, \P_Nu^N), \P_Nu^N)\\
& \quad =: A_1 + A_2 + A_3 + A_4
\end{split}
\label{X1}
\end{align}

\noi
and
\begin{align}
\begin{split}
 \N_1(& \Psi, u) - \N_1^N(\Psi, u^N)
  = \N_1(\Psi, u) - \P_N \N_1(\P_N\Psi, \P_Nu^N)\\
& \quad = \N_1(\P_N^\perp \Psi, u\big)   + \N_1(\P_N\Psi, u - u^N)\\
& \quad \quad  + \N_1(\P_N\Psi , \P_N^\perp u^N)
  +\P_N^\perp \N_1(\P_N\Psi, \P_Nu^N)\\
& \quad =: B_1 + B_2 + B_3 + B_4.
\end{split}
\label{X2}
\end{align}

\noi
Similar expressions hold
for the differences
\begin{align}
\N_2(u, \N(u, u)) - \N_2^N(u^N, \N^N(u^N, u^N))
\label{X1a}
\end{align}
and 
\begin{align}
 \N_2(u,  \Psi) - \N_2^N(u^N, \Psi).
\label{X2a} 
\end{align}

Let us first estimate \eqref{X2}.
Given $N \in \NB$, define $\wt L_{\o, N}^\perp(T)$ by 
\begin{align}
\wt   L_{\omega, N}^\perp(T)&= 
  \|\ind_{[0, T]}\P_N^\perp \Psi\|_{X^{-\frac 12 -  \dl, \frac 12 -\dl}}
+ 
\|\ind_{[0, T]}\P_N^\perp \Psi\|_{Y^{-\frac 12 -  \dl, \frac {11}{16} +\dl}_{2, 4}}.
\label{F3}
\end{align}

\noi
See \eqref{BX4} below.
Then, from \eqref{F1} and \eqref{F3}, we have
\begin{align}
\wt   L_{\omega, N}^\perp(T) \les N^{-\frac \dl 2}
  L_{\omega}(T).
  \label{F4}
\end{align}

From   the estimates in 
 Appendix~\ref{SEC:B}, 
\eqref{F4}, and \eqref{embed3} (see also \eqref{X5a} below),  
we have 
\begin{align}
\begin{split}
& \| B_1 + B_2 + B_3\|_{X^{-\al, 1- \al, T_1}}\\
& \quad 
\les T_1^\ta \wt L_{\o, N}^\perp (T) \| u\|_{X^{- (1- \al), \al, T_1}}\\
& \quad \quad + T_1^\ta L_{\o} (T) \Big(\| u - u^N \|_{X^{- (1- \al), \al, T_1}}
+   \| \P_N^\perp u^N \|_{X^{- (1- \al), \al, T_1}}\Big)\\
& \quad \les T_1^\ta L_{\o} (T) \| u - u^N \|_{X^{- \al, \al, T_1}_{p, 2}}
+ N^{-\frac\dl 2}
T_1^\ta
L_{\o} (T)
\Big(  \| u \|_{X^{-  \al, \al, T_1}_{p, 2}}
+   \|  u^N \|_{X^{-  \al, \al, T_1}_{p, 2}}\Big)
\end{split}
\label{X3}
\end{align}

\noi
and
\begin{align}
\begin{split}
\| B_4\|_{X^{-\al, 1- \al, T_1}}
& \les N^{-\frac {\dl}{2}}
\| B_4\|_{X^{-\al + \frac \dl 2, 1- \al, T_1}}
 \les N^{-\frac {\dl}{2}}T_1^\ta L_{\o} (T) \|  u^N \|_{X^{- \frac 12 - \frac \dl 2, \al, T_1}}\\
& \les N^{-\frac {\dl}{2}} T_1^\ta L_{\o} (T) \|  u^N \|_{X_{p, 2}^{-\al, \al, T_1}}
\end{split}
\label{X4}
\end{align}

\noi
for some small $\ta > 0$.
Therefore, from \eqref{X3},  \eqref{X4}, 
and the symmetry between $\N_1$ and $\N_2$, we have
\begin{align}
\begin{split}
\| \eqref{X2} + \eqref{X2a}\|_{X^{-\al, 1- \al, T_1}}
&  \les T_1^\ta L_{\o} (T) \| u - u^N \|_{X^{- \al, \al, T_1}_{p, 2}}\\
& \quad + N^{-\frac\dl 2}
T_1^\ta
L_{\o} (T)
\Big(  \| u \|_{X^{-  \al, \al, T_1}_{p, 2}}
+   \|  u^N \|_{X^{-  \al, \al, T_1}_{p, 2}}\Big).
\end{split}
\label{X5}
\end{align}

Next, we estimate  the terms in \eqref{X1}.
The main nonlinear analysis comes from 
\cite[(2.27)-(2.59)
pp.\,125-130]{BO3}
and \cite[``Estimate on (i)'' on pp.\,295-296]{OH4}.
Here,  the latter replaces
\cite[Estimation of (2.62) on p.\,131]{BO3}, 
where the a priori assumption \eqref{BOO} was used.
In \cite{BO3}, 
the nonlinear analysis (\cite[(2.27)-(2.59)
pp.\,125-130]{BO3})
was estimated by  
the 
$X^{-(1-\al),\al}$-norm of~$u$. 
In particular, in estimating the terms
with $\P_N^\perp u^N$
 in \eqref{X1} (namely, the first and third terms on the right-hand side of~\eqref{X1}), 
we can apply \eqref{embed3}
to  gain a negative power of $N$ as follows:
\begin{align}
\begin{split}
\| \P_N^\perp u^N \|_{X^{-(1-\al),  \al, T_1}}
& \les 
N^{-\frac \dl 2} \| \P_N^\perp u^N \|_{X^{-\frac 12 -  \frac 12 \dl ,  \al, T_1}}\\
& \les N^{-\frac \dl 2}\|  u^N \|_{X_{p, 2}^{-\al,  \al, T_1}}.
\end{split}
\label{X5a}
\end{align}

\noi
As for and \cite[``Estimate on (i)'' on pp.\,295-296]{OH4}
on $R_\al$ in \cite[(58)]{OH4}, 
we used $\jb{n}^{-1- \al} \le \jb{n}^{-3\al}$.
This  can be replaced by 
$\jb{n}^{-1- \al + \frac \dl 2} \le \jb{n}^{-3\al}$,
which allows us to gain $N^{-\frac\dl2}$
from~$\P_N^\perp$.

From the discussion above, 
a straightforward modification of the estimates in 
\cite[(2.27)-(2.59)
pp.\,125-130]{BO3}
and \cite[``Estimate on (i)'' on pp.\,295-296]{OH4}
yields
\begin{align}
\begin{split}
 \|  A_1 & \|_{X^{-\al, 1- \al, T_1}}\\
& \les T_1^\ta 
\| \N_1(u, u) - \P_N\N_1(\P_Nu^N, \P_Nu^N)\|
_{X^{-\al, 1- \al, T_1}}
\| u \|_{X^{-(1-\al), \al, T_1}}\\
& \quad  + T_1^\ta
\Big( \|  u \|_{X^{- (1- \al), \al, T_1}}^2 +  \|  u^N \|_{X^{-  (1- \al), \al, T_1}}^2\Big)
 \| u -  u^N \|_{X^{-   (1- \al), \al, T_1}}\\
& \quad  + N^{-\frac \dl 2}T_1^\ta
 \|  u^N \|_{X^{-   (1- \al), \al, T_1}}^3\\
& \les T_1^\ta 
\| \N_1(u, u) - \N^N_1(u^N, u^N)\|
_{X^{-\al, 1- \al, T_1}}
 \|  u \|_{X^{-  \al, \al, T_1}_{p, 2}}\\
& \quad  + T_1^\ta
\Big( \|  u \|_{X^{-  \al, \al, T_1}_{p, 2}}^2 +  \|  u^N \|_{X^{-  \al, \al, T_1}_{p, 2}}^2\Big)
 \| u -  u^N \|_{X^{-  \al, \al, T_1}_{p, 2}}\\
& \quad  + N^{-\frac \dl 2}T_1^\ta
 \|  u^N \|_{X^{-  \al, \al, T_1}_{p, 2}}^3.
\end{split}
\label{X6a}
\end{align}

\noi
Here,  the first term on the right-hand side comes from 
\cite[(II.1) on pp.\,126-127]{BO3}, 
while the second and third terms on the right-hand side
come from estimating the other cases
 trilinearly,  using
\begin{align*}
\N(u, u) - \P_N\N(\P_Nu^N, \P_Nu^N) 
& = \N(u, u) -\N(u^N, u^N)\\
& \quad 
+ \N( u^N, \P_N^\perp u^N)
+ \N(\P_N^\perp u^N, \P_Nu^N)\\
& \quad 
+  \P_N^\perp \N(\P_Nu^N, \P_Nu^N).
\end{align*}

\noi
Similarly, we have 
\begin{align}
\begin{split}
 \|  A_2  \|_{X^{-\al, 1- \al, T_1}}
& \les T_1^\ta 
\| \N^N_1(u^N, u^N)\|_{X^{-\al, 1- \al, T_1}}
 \| u -  u^N \|_{X^{-  \al, \al, T_1}_{p, 2}}\\
& \quad  
+ T_1^\ta
 \|  u^N \|_{X^{-  \al, \al, T_1}_{p, 2}}^2
  \| u -  u^N \|_{X^{-  \al, \al, T_1}_{p, 2}}
\end{split}
\label{X6b}
\end{align}

\noi
and 
\begin{align}
\begin{split}
 \|  A_3  \|_{X^{-\al, 1- \al, T_1}}
& \les 
N^{-\frac \dl 2}T_1^\ta
\| \N^N_1(u^N, u^N)\|_{X^{-\al, 1- \al, T_1}}
 \| u^N \|_{X^{-  \al, \al, T_1}_{p, 2}}\\
& \quad  
+ N^{-\frac \dl 2}T_1^\ta
 \|  u^N \|_{X^{-  \al, \al, T_1}_{p, 2}}^3 .
\end{split}
\label{X6c}
\end{align}

\noi
In handling the term $A_4$ in \eqref{X1}
with $\P_N^\perp$ outside the nonlinearity
we simply 
use 
\begin{align*}
\jb{n}^\frac \dl 2 \les \jb{n_1}^\frac \dl 2\jb{n_2}^\frac \dl 2
\qquad \text{and}
\qquad 
\jb{n}^\frac \dl 2 \les \jb{n_2}^\frac \dl 2\jb{n_3}^\frac \dl 2\jb{n_4}^\frac \dl 2, 
\end{align*}

\noi
where
$n_3$ and $n_4$ are the spatial frequencies
of the first and second factors of 
$\P_N\N(\P_Nu^N, \P_Nu^N)$ in $A_4$; see also \eqref{X4} above.
Thus, we have
\begin{align}
\begin{split}
 \|  A_4  \|_{X^{-\al, 1- \al, T_1}}
& \les 
 N^{-\frac \dl 2} \|  A_4  \|_{X^{-\al+ \frac \dl 2, 1- \al, T_1}}\\
& \les 
N^{-\frac \dl 2}T_1^\ta
\| \N^N_1(u^N, u^N)\|_{X^{-\al, 1- \al, T_1}}
 \| u^N \|_{X^{-  \al, \al, T_1}_{p, 2}}\\
& \quad  
+ N^{-\frac \dl 2}T_1^\ta
 \|  u^N \|_{X^{-  \al, \al, T_1}_{p, 2}}^3.
\end{split}
\label{X6d}
\end{align}

\noi
Here,  the first term on the right-hand side of \eqref{X6d}
comes from \cite[(II.1) on pp.\,126-127]{BO3}, 
where we used the fact that 
$\jb{n_1}^{2\al - 1} = \jb{n_1}^{-2\dl}$.
(In \cite{BO3}, in view of $2\al - 1< 0$, 
this factor $\jb{n_1}^{2\al - 1}$ was simply thrown away; see \cite[(2.37)]{BO3}.)
Hence, from 
\eqref{X6a}, \eqref{X6b},  \eqref{X6c}, \eqref{X6d}, and 
the symmetry between $\N_1$ and $\N_2$, we obtain
\begin{align}
\begin{split}
\| \eqref{X1} + \eqref{X1a} \|_{X^{-\al, 1- \al, T_1}}
& \les T_1^\ta 
\| \N_1(u, u) - \N_1^N(u^N, u^N)\|
_{X^{-\al, 1- \al, T_1}}
 \|  u\|_{X^{-  \al, \al, T_1}_{p, 2}}\\
& \quad  + T_1^\ta
\Big( \|  u \|_{X^{-  \al, \al, T_1}_{p, 2}}^2 +  \|  u^N \|_{X^{-  \al, \al, T_1}_{p, 2}}^2\Big)
 \| u -  u^N \|_{X^{-  \al, \al, T_1}_{p, 2}}\\
& \quad  +
 T_1^\ta 
\| \N_1^N(u^N, u^N)\|_{X^{-\al, 1- \al, T_1}}
 \| u -  u^N \|_{X^{-  \al, \al, T_1}_{p, 2}}\\
& \quad  +
N^{-\frac \dl 2}T_1^\ta
\| \N_1^N(u^N, u^N)\|_{X^{-\al, 1- \al, T_1}}
 \| u^N \|_{X^{-  \al, \al, T_1}_{p, 2}}\\
& \quad  + N^{-\frac \dl 2}T_1^\ta
 \|  u^N \|_{X^{-  \al, \al, T_1}_{p, 2}}^3.
\end{split}
\label{X7}
\end{align}

Given $R \geq 1$, 
by choosing $T_1 = T_1(R)> 0$ sufficiently small 
such that  the condition~\eqref{nonlin1a} is satisfied.
Then, by possibly making $T_1 = T_1(R)> 0$ 
small, it follows from~\eqref{X5}
and~\eqref{X7} with \eqref{nonlin2}
that 
\begin{align}
\begin{split}
& \sum_{j = 1}^2 \| \N_j(u, u) - \N^N_j(u^N, u^N)\|_{X^{-\al, 1- \al, T_1}}\\
& \quad \les 
 T_1^\ta
\Big( \|  u \|_{X^{-  \al, \al, T_1}_{p, 2}}^2 +  R^3 + L_\o(T) R \Big)
 \| u -  u^N \|_{X^{-  \al, \al, T_1}_{p, 2}}\\
& \quad \quad  +
N^{-\frac \dl 2}T_1^\ta
\Big( R^4   +  
L_\o(T) R^2 \Big)
\end{split}
\label{X8}
\end{align}

\noi
under an extra assumption $u$:
\begin{align}
 \|u\|_{X^{-\alpha,\alpha,T_1}_{p,2}} \le 2R.
\label{XX1}
\end{align}


\noi
As mentioned in Section \ref{SEC:LWP}, 
the temporal regularity on the left-hand side
of \eqref{X8}
is $b = 1- \al = \frac 12 + \dl > \frac 12$,
which is used in \eqref{X12} below.

The following estimate follows from a slight modification 
of the bilinear estimate \eqref{bilin1}
(see
\cite[(I.1) and (I.2) on pp.\,122-125]{BO3}
and \eqref{embed0a}):
\begin{align}
\begin{split}
\| \N_0 & (u, u)-\N_0^N(u^N, u^N)\|_{X^{-\al, \al, T_1}}\\
& \les T_1^\ta
\Big(  \| u  \|_{X^{-  (1-\al), \al, T_1}}
+  \|   u^N \|_{X^{-  (1-\al), \al, T_1}}\Big)
 \| u -  u^N \|_{X^{-  (1-\al), \al, T_1}}\\
& \les T_1^\ta
\Big(  \| u  \|_{X^{-  \al, \al, T_1}_{p, 2}}
+  \|   u^N \|_{X^{-  \al, \al, T_1}_{p, 2}}\Big)
 \| u -  u^N \|_{X^{-  \al, \al, T_1}_{p, 2}}.
\end{split}
\label{X9}
\end{align}

\noi
As for the difference of the linear solutions, 
it follows from 
\eqref{time1} (with $T_1 \le 1$) and 
\eqref{lin1} that 
\begin{align*}
\| S(t) u(0) - S(t) u^N(0) \|_{X_{p, 2}^{-\al, \al, T_1}}
& \le 
\| S(t) u(0) - S(t) u^N(0) \|_{X_{p, 2}^{-\al, \al, 1}}\\
&  \les \| u(0) - u^N(0) \|_{\ft b^{-\al}_{p, \infty}}.
\end{align*}

\noi
Therefore, putting 
\eqref{DD7a}, \eqref{NN2}, 
\eqref{DD8}, \eqref{X8}, and \eqref{X9}
together 
 we obtain
\begin{align}
\begin{split}
\| u - u^N \|_{X_{p, 2}^{-\al, \al, T_1}}
&  \le  D_0  \| u(0) - u^N(0) \|_{\ft b^{-\al}_{p, \infty}}\\
& \quad 
+  D_1 T_1^\ta
\Big( \|  u \|_{X^{-  \al, \al, T_1}_{p, 2}}^2 +  R^3 + L_\o(T) R \Big)
 \| u -  u^N \|_{X^{-  \al, \al, T_1}_{p, 2}}\\
&  \quad  +
D_2
N^{-\frac \dl 2}T_1^\ta
\Big( R^4   +  
L_\o(T) R^2 \Big)
\end{split}
\label{X10}
\end{align}

\noi
under the assumptions \eqref{nonlin1a} and \eqref{XX1}.
Here, we took general initial data $u(0)$ and $u^N(0)$
so that we can apply the estimate \eqref{X10}
to a general time interval of length $T_1$.

Next, let us bound the difference 
of $u$ and $u^N$ in the $C([0, T_1]; \ft b^{-\al}_{p, \infty}(\T))$-norm.
A bilinear version of \eqref{nonlin3} yields
\begin{align}
\begin{split}
\| \N_0 &  (u, u)-\N_0^N(u^N, u^N) \|_{C([0, T_1]; \ft b^{-\al}_{p, \infty})}\\
& \les T_1^\ta
\Big(  \| u  \|_{X^{-  \al, \al, T_1}_{p, 2}}
+  \|   u^N \|_{X^{-  \al, \al, T_1}_{p, 2}}\Big)
 \| u -  u^N \|_{X^{-  \al, \al, T_1}_{p, 2}}.
\end{split}
\label{X11}
\end{align}	

\noi
Hence, from \eqref{X8} and \eqref{X11}, we have\footnote{In general, 
the constants $D_1$ and $D_2$ in \eqref{X10} and \eqref{X12}
are different, but we simply take the worse ones.}
\begin{align}
\begin{split}
\| u - u^N \|_{C([0, T_1]; \ft b^{-\al}_{p, \infty})}
&  \le \| u(0) - u^N(0) \|_{\ft b^{-\al}_{p, \infty}}\\
& \quad 
+  D_1 T_1^\ta
\Big( \|  u \|_{X^{-  \al, \al, T_1}_{p, 2}}^2 +  R^3 + L_\o(T) R \Big)
 \| u -  u^N \|_{X^{-  \al, \al, T_1}_{p, 2}}\\
&  \quad  +
D_2
N^{-\frac \dl 2}T_1^\ta
\Big( R^4   +  
L_\o(T) R^2 \Big)
\end{split}
\label{X12}
\end{align}

\noi
under the assumptions \eqref{nonlin1a} and \eqref{XX1}.
We point out that the estimates \eqref{X10} and \eqref{X12}
hold true on a general time interval of length $T_1$.

\medskip

\noi
$\bullet$ {\bf Step 2:}
Fix $T \gg 1$ and $0 < \eps \ll 1$.
We now establish the difference estimate \eqref{XX2} on the time interval $[0, T]$
by iterating the local-in-time estimates \eqref{X10}
and \eqref{X12}
with the probabilistic input from Proposition~\ref{PROP:main}.

Given $N \in \NB$, let $\O_{T,  \eps}(N)
= \O_1 \cap \cdots \cap \O_4$ be as in 
Proposition~\ref{PROP:main}, 
where $\O_k$, $k = 1, \dots, 4$, 
are as in \eqref{O1}, \eqref{O3}, \eqref{O5}, 
and \eqref{O10}, respectively.
In particular, if necessary, we have made
$T_1$ smaller such that \eqref{O8} is satisfied.
In the following, it is understood that we work on 
$\O_{T,  \eps}(N)$
and that all the estimates
are restricted to $\O_{T,  \eps}(N)$, where the value of $N$ may increase in each step.

For now, 
assume that 
\begin{align}
\|  u \|_{X^{-  \al, \al}_{p, 2}(I_j)}
\le \|  u^N \|_{X^{-  \al, \al}_{p, 2}(I_j)} + 1
\les K_1
\label{Y1}
\end{align}

\noi
for $I_j = [j T_1, (j+1) T_1] \cap [0, T]$, $j = 0, 1, \dots, \big[\frac T {T_1}\big]$,
where the second inequality follows from~\eqref{O9a}.
Note that, with $R = C_*(2K_1 + 1)$, 
 \eqref{nonlin1a} (on the interval $I_j$) and \eqref{Y1} 
 (see also \eqref{O7} and \eqref{O9a}) implies
\eqref{XX1} (on the interval~$I_j$).
Then, 
in view of \eqref{X10} and \eqref{X12} with \eqref{O7}
(see also~\eqref{bd1} in Lemma \ref{LEM:nonlin2})
we further impose that $T_1 > 0$ be sufficiently small such that 
\begin{align}
\begin{split}
  T_1^\ta
\Big(  K_1^3 + K_1 K_2\Big) \ll 1, \\
T_1^\ta
\Big( K_1^4   +  
K_1^2 K_2 \Big) \ll 1.
\end{split}
\label{Y2}
\end{align}

\noi
In the following, we work iteratively on each interval $I_j$
and verify \eqref{Y1}.

Let us now consider the first time interval $I_0 = [0, T_1]$.
By the local well-posedness theory (see \eqref{nonlin1}), 
there exists small $T_0 > 0$ such that 
\begin{align}
\|  u \|_{X^{-  \al, \al, T_0}_{p, 2}}
\les K_1.
\label{Y3}
\end{align}

\noi
Then, from 
\eqref{X10} (but with $T_0$ replacing $T_1$
and with $u(0) = u^N(0)$) with  \eqref{Y2} and \eqref{Y3}, we have 
\begin{align*}
\| u - u^N \|_{X_{p, 2}^{-\al, \al, T_0}}
&  \le
\frac 12   \| u -  u^N \|_{X^{-  \al, \al, T_0}_{p, 2}}
    +
N^{-\frac \dl 2}
\end{align*}

\noi
Hence, we have 
\begin{align*}
\| u - u^N \|_{X_{p, 2}^{-\al, \al, T_0}}
&  \le
2N^{-\frac \dl 2}.
\end{align*}

\noi
Therefore, 
by a standard continuity argument (see also Remark \ref{REM:cont}), 
we  conclude that there exists $N_0 \in \NB$
such that \eqref{Y1}  holds
on the entire time interval $I_0 = [0, T_1]$
for  any  $N  \ge N_0$.
As a result, we obtain 
\begin{align}
\| u - u^N \|_{X_{p, 2}^{-\al, \al}(I_0)}
&  \le
2N^{-\frac \dl 2}
\label{Y6}
\end{align}

\noi
for any $N \ge N_0$.
By applying \eqref{Y2} and \eqref{Y6} (with \eqref{O7})
to  \eqref{X12}, we then obtain 
\begin{align}
\| u - u^N \|_{C(I_0; \ft b^{-\al}_{p, \infty})}
&  \le 2 N^{-\frac \dl 2}
\label{Y7}
\end{align}

\noi
for any $N \ge N_0$.

On the second interval $I_1 = [T_1, 2T_1]\cap [0, T]$, we repeat an analogous analysis.
From~\eqref{Y7}, we have 
\begin{align*}
\| u(T_1) - u^N(T_1)  \|_{\ft b^{-\al}_{p, \infty}}
&  \le 2 N^{-\frac \dl 2}
\end{align*}

\noi
for any $N \ge N_0$.
By the local theory, there exists small $T_0 > 0$ such that 
\begin{align*}
\|  u \|_{X^{-  \al, \al}_{p, 2}([T_1, T_1 + T_0])}
\les K_1.
\end{align*}

\noi
Then, from 
\eqref{X10} (but on $[T_1, T_1+ T_0]$) with  \eqref{Y2} and \eqref{Y3}, we have 
\begin{align*}
\| u - u^N \|_{X_{p, 2}^{-\al, \al}([T_1, T_1 + T_0])}
&  \le 2 N^{-\frac \dl 2}+ 
\frac 12   \| u -  u^N \|_{X^{-  \al, \al}_{p, 2}([T_1, T_1 + T_0])}
    +
N^{-\frac \dl 2}, 
\end{align*}

\noi
which yields
\begin{align*}
\| u - u^N \|_{X_{p, 2}^{-\al, \al}([T_1, T_1 + T_0])}
&  \le
6N^{-\frac \dl 2}.
\end{align*}

\noi
Therefore, 
it follows from  a standard continuity argument that 
 there exists $N_1 \in \NB$
such that \eqref{Y1}  holds
on the entire time interval $I_1$
for any $N \ge N_1$.
As a result, we obtain 
\begin{align}
\| u - u^N \|_{X_{p, 2}^{-\al, \al}(I_1)}
&  \le
6N^{-\frac \dl 2}
\label{Y9a}
\end{align}

\noi
for any $N \ge N_1$.
Hence, from \eqref{X12}, \eqref{Y1}  \eqref{Y2}, and \eqref{Y9a}, we obtain
\begin{align*}
\| u - u^N \|_{C(I_1; \ft b^{-\al}_{p, \infty})}
&  \le 2 N^{-\frac \dl 2} + \frac 12 \cdot 6 N^{-\frac \dl 2} +  N^{-\frac \dl 2}
= 6 N^{-\frac \dl 2}.
\end{align*}

\noi
for any $N \ge N_1$.

Proceeding iteratively, 
we conclude that, on the $j$th interval $I_j = [jT_1, (j+1) T_1]\cap [0, T]$, 
$j = 0, 1, \dots, \big[\frac T {T_1}\big]$,
there exists $N_j \in \NB$ such that 
\begin{align}
\begin{split}
\| u - u^N \|_{X_{p, 2}^{-\al, \al}(I_j)}
&  \le
\bigg( \sum_{k = 0}^j 2^{k+1}\bigg) N^{-\frac \dl 2}, \\
\| u - u^N \|_{C(I_j; \ft b^{-\al}_{p, \infty})}
&  \le\bigg( \sum_{k = 0}^j 2^{k+1}\bigg)  N^{-\frac \dl 2}
\end{split}
\label{Y10}
\end{align}

\noi
for any $N \ge N_j$.
Note that $T_1$ depends only on $T$ and $\eps$;
see 
\eqref{O8} and \eqref{Y2} with 
\eqref{O1a} and \eqref{O5a}.
See also \eqref{nonlin1}
with $R \les K_1$ as in \eqref{O7}.
Therefore, by setting 
\[N_* = N_*(T, \eps)  = N_{[T/T_1]}
\qquad \text{and}\qquad
\O_{T, \eps} = 
\O_{T, \eps}(N_*(T, \eps)), \]

\noi
where the latter is as in Proposition \ref{PROP:main},
we conclude from \eqref{Y10}
that, on $\O_{T, \eps}$, we have 
\begin{align*}
\| u - u^{N_*} \|_{C([0, T]; \ft b^{-\al}_{p, \infty})}
&  \le C(T, \eps)  N_*^{-\frac \dl 2}.
\end{align*}

\noi
This concludes the proof of Proposition \ref{PROP:GWP1}.
\end{proof}

We now present the proof of Theorem \ref{THM:1}.
We first note that the claimed almost sure global well-posedness
of SKdV \eqref{KdV1} with the white noise initial data immediately follows from  the 
`almost' almost sure global well-posedness
result established in Proposition \ref{PROP:GWP1}; 
see \cite{CO, BOP2}.
Indeed, define $\Si \subset \O$ by 
\begin{align}
\Si = \bigcup_{k = 1}^\infty
\bigcap_{j=1}^\infty \Omega_{2^j, \frac{1}{k2^j} }, 
\label{Z1}
\end{align}

\noi
where $\O_{T, \eps}$ is as in 
Proposition \ref{PROP:GWP1}.
Then, we have 
\begin{align*}
\PP(\Si^c) \le \inf_{k \in \NB}
\sum_{j = 1}^\infty \PP(\Omega_{2^j, \frac{1}{k2^j} }^c) 
= \inf_{k \in \NB}\frac 1k = 0.
\end{align*}

\noi
Moreover, if $\o \in \Si$, 
then there exists $k \in \NB$ such that $\o \in \Omega_{2^j, \frac{1}{k2^j} }$
for any $j \in \NB$, 
which implies that the corresponding solution $u = u(\o)$
to SKdV \eqref{KdV1} exists globally in time.

It remains to prove \eqref{th1}.
It follows from the proof of Proposition \ref{PROP:GWP1} that, 
on $\O_{T, \eps}= \O_{T, \eps}(N_*(T, \eps))$, we have 
  \begin{align}
\sup_{t\in [0,T]}\| u^N(t) - u^{N_*}(t)\|_{\widehat{b}_{p,\infty}^{-\alpha}}
\le C(T, \eps) N_*^{-\frac{\dl}{2}}.
\label{Z2}
\end{align}

\noi
for any $N \ge N_*$.
Define $\wt \O_1(N)= \wt \O_1(T, \eps, N) \subset \O$ by 
 \begin{align}
\wt  \Omega_1(N)= \bigcap_{j=0}^{[T/T_1]}
 \Big\{\|u^N(j T_1 )\|_{\widehat{b}^{-\alpha}_{p,\infty}}\le 2K_1\Big\}.
\label{Z3}
 \end{align}
 
\noi
Namely, we replaced $K_1$ in 
\eqref{O1} by $2K_1$.
By taking $N_*$ sufficiently large, 
it follows from~\eqref{Z2} that 
$\O_{T, \eps} \subset \wt \O_1(N)$
for any $N \ge N_*$.
Hence, by setting
$\wt \Omega_{T,\eps}(N) = \wt \O_1\cap \O_2 \cap \O_3 \cap \O_4$, 
where 
$\O_2$, $\O_3$, and $\O_4$
are 
\eqref{O3}, \eqref{O5}, and \eqref{O10}, respectively, 
we have 
\begin{align}
\O_{T, \eps} \subset \wt \Omega_{T,\eps}(N)
\label{Z3a}
\end{align}

\noi
for any $N \ge N_*$.
Now, by repeating Step 2 in the proof of Proposition \ref{PROP:GWP1},\footnote{In \eqref{Z3}, 
we replaced $K_1$ by $2K_1$, which worsens  constants in the argument.
We can, however, implement the proof of Proposition \ref{PROP:GWP1}
to incorporate these worse constants from the beginning.} 
we conclude that, there exists $N_{**} = N_{**}(T, \eps) \in \NB$
such that, 
on $\O_{T, \eps}$, we have 
  \begin{align}
\sup_{t\in [0,T]}\| u(t) - u^{N}(t)\|_{\widehat{b}_{p,\infty}^{-\alpha}}
\le C(T, \eps) N^{-\frac{\dl}{2}}.
\label{Z3b}
\end{align}

\noi
for any $N \ge N_{**}$. 
This in particular implies that, 
for each $\o \in \O_{T, \eps}$, 
the solution $u = u(\o)$ to SKdV \eqref{KdV1} is the limit of $u^N = u^N(\o)$
in $C([0, T]; \widehat{b}_{p,\infty}^{-\alpha}(\T))$.
Hence, given $t \in \R_+$, 
it follows from the discussion above 
that, for each $\o \in \Si$, 
\begin{align*}
\|u^N(t; \o)-u(t;\o)\|_{\widehat{b}^{-\alpha}_{p,\infty}}\too 0
\end{align*}

\noi
as $N \to \infty$.
This in particular implies convergence in law of $u^N(t)$ to $u(t)$.
Recalling that $\Law (u^N(t)) = \mu_{1+t}$ for any $N \in \NB$, 
we then conclude that 
\[\Law (u(t)) = \mu_{1+t}.\]

\noi
This concludes the proof of Theorem \ref{THM:1}.

\begin{remark}\label{REM:bound}\rm
Let $\o \in \O_{T, \eps}$.
Then, from 
\eqref{Z3a}, \eqref{Z3b}, and 
Proposition \ref{PROP:main}, 
we have
\begin{align}
\sup_{t\in [0,T]}\|u(t)\|_{\widehat{b}_{p,\infty}^{-\alpha}}&\le  C
\sqrt {\log \frac 1 \eps}\sqrt{T \log T}.
\label{Z5}
\end{align}

\noi
Fix $ k \in \NB$, 
and suppose that 
$\o \in \bigcap_{j=1}^\infty \Omega_{2^j, \frac{1}{k2^j} }$.
Then, from \eqref{Z5}, we obtain
\begin{align*}
\|u(t)\|_{\widehat{b}_{p,\infty}^{-\alpha}}&\le  C
\sqrt {\log k}\sqrt{1+t} \log (1+t)
\end{align*}

\noi
for any $t \in \R_+$.
Namely, we have
\begin{align}
\|u(t)\|_{\widehat{b}_{p,\infty}^{-\alpha}}&\le  C(\o)\sqrt{1+t} \log (1+t)
\label{Z6}
\end{align}

\noi
for any $t \in \R_+$ and $\o \in \Si$.
Note that the growth bound \eqref{Z6} is not optimal, 
and 
we can improve it by modifying the definition \eqref{Z1} of $\Si$.
For example, by redefining $\Si$ by 
\begin{align*}
\Si = \bigcup_{k = 1}^\infty
\bigcap_{j=1}^\infty \Omega_{2^j, \frac{1}{kj^2} }
\end{align*}

\noi
and repeating the argument, 
we obtain the following growth bound:
\begin{align*}
\|u(t)\|_{\widehat{b}_{p,\infty}^{-\alpha}}&\le  C(\o)\sqrt{1+t} \sqrt{\log (1+t)} \sqrt{\log \log (1+t)}.
\end{align*}

\noi
In this way,  we can obtain a growth bound which is only slightly faster
than $\sqrt{t \log t}$, $t \gg 1$
(but the random constant $C(\o)$ gets worse).

\end{remark}

\appendix

\section{Growth bound on the stochastic convolution for large times}
\label{SEC:A}

In this appendix, we present the proof of Lemma \ref{LEM:sto1}.

\begin{proof}[Proof of Lemma \ref{LEM:sto1}]
Fix 
 $s < 0$ and $1 \le p, q < \infty$ such that $s p < -1$, 
 and $(b-1)q < -1$.
 We also fix $1 \le r  < \infty$ and $T\ge 1$.
 Without loss of generality, we assume 
 \begin{align}
 r  \ge \max(p, q).
\label{AP00}
 \end{align}
 
 \noi
Before proceeding further, we first recall the following bound
for a Gaussian random variable~$g$:
\begin{align}
 \| g\|_{L^r (\O)} \les \sqrt r  \|g\|_{L^2(\O)}.
\label{AP0}
\end{align}

\medskip

\noi
(i) Let $I = [t_0, t_1] \subset [0, T]$ 
be an interval of length $|I|\le 1$.
The first inequality in~\eqref{sto1}
follows from \eqref{FL3}, 
and thus we focus on proving the second inequality in \eqref{sto1}.

Recall that
\begin{align}
 \| u \|_{Y^{s, b}_{p, q}}
= \| S(-t) u(t) \|_{\F L^{s, p}_x \F L^{b, q}_t}, 
\label{AP0a}
\end{align}

\noi
where $\F L^{b, q}_t$ and 
$\F L^{s, p}_x$
are  the Fourier-Lebesgue spaces defined in \eqref{FL1}
and \eqref{FL2}, respectively.
Let $\Phi(t) = S(-t) \Psi(t)$ be the interaction representation 
of $\Psi$.
From \eqref{psi1}
with~\eqref{W1}, we have
\[ \ft{\ind_I \Phi}(n, t) = \ind_I(t) \int_0^t e^{-it'n^3} d \be_n(t').\]

\noi
By taking the temporal Fourier transform, we then have
\begin{align}
\begin{split}
\ft{\ind_I \Phi}(n, \tau) 
& = \int_{t_0}^{t_1} e^{-it\tau} \int_0^t e^{-it'n^3} d \be_n(t')dt\\
& = \int_0^{t_1} e^{-it'n^3} \int_{\max(t_0, t')}^{t_1}e^{-it\tau} dt d \be_n(t').
\end{split}
\label{AP1}
\end{align}

\noi
The inner integral can be estimated as
\begin{align}
\bigg|\int_{\max(t_0, t')}^{t_1}e^{-it\tau}dt\bigg|
\les \min \bigg(1, \frac{1}{|\tau|}\bigg) \les \frac 1{\jb\tau}.
\label{AP2}
\end{align}

From \eqref{time1} (for the $Y^{s, b}_{p, q}$-space) and \eqref{AP0a}, we have 
\begin{align}
\begin{split}
\| \Psi \|_{Y^{s, b}_{p, q}(I)}
& \le \| \ind_I \Psi \|_{Y^{s, b}_{p, q}(I)}
= \| \jb{n}^s \jb{\tau}^b \ft{\ind_I \Phi}(n, \tau) \|_{\l^p_n  L^q_\tau}.
\end{split}
\label{AP3}
\end{align}

\noi
Then, by \eqref{AP3},  Minkowski's integral inequality, and  \eqref{AP0}
followed by the Ito isometry with \eqref{AP1},  \eqref{AP2} and $t_1 \le T$, we have 
\begin{align*}
\Big\| \| \Psi \|_{Y^{s, b}_{p, q}(I)}\Big\|_{L^r (\O)} 
& = \Big\|\| \jb{n}^s \jb{\tau}^b \ft{\ind_I \Phi}(n, \tau) \|_{\l^p_n  L^q_\tau}\Big\|_{L^r (\O)}\\
& \le \Big\|\| \jb{n}^s \jb{\tau}^b \ft{\ind_I \Phi}(n, \tau)\|_{L^r (\O)} \Big\|_{\l^p_n  L^q_\tau}\\
& \les \sqrt r  \Big\|\| \jb{n}^s \jb{\tau}^b \ft{\ind_I \Phi}(n, \tau)\|_{L^2(\O)} \Big\|_{\l^p_n  L^q_\tau}\\
& \les \sqrt {r  T} \| \jb{n}^s \jb{\tau}^{b-1} \|_{\l^p_n  L^q_\tau}\\
& \les \sqrt {r  T}, 
\end{align*}

\noi
since $sp < -1$ and $(b-1) q < -1$.
This proves \eqref{sto1}.

\medskip

\noi (ii)
It follows from 
\cite[Proposition 4.5]{OH4}
that the stochastic convolution is continuous in time with values in $\ft b^{s}_{p, \infty}(\T)$
when $sp < -1$, at least locally in time.
In the following, we estimate its growth in a direct manner
by following the argument in 
\cite[Lemma 3.4]{OP}.

Without loss of generality, assume that $T \in 2^\NB$.
For an integer $k \in \Z\cap [-\log_2 T, \infty)$, 
let $\{ t_{\l, k}:\l = 0, 1,\dots, 2^k T\} $ be $2^{k}T+1$ equally spaced points on $[0, T]$, 
i.e.~$t_{0,k}=0$ and $t_{\l, k} - t_{\l-1, k} = 2^{-k}$ for $\l =1, \dots, 2^kT$.
Let $\Phi(t) = S(-t) \Psi(t)$ be the interaction representation 
of $\Psi$.
Then, given $t \in [0, T]$, 
it follows from the continuity (in time) of $\Psi$ 
and $\Psi(0) = 0$ that 
\begin{align}
\Phi(t)
= \sum_{k = -\log_2 T}^\infty \big(\Phi(t_{\l_k, k}) - \Phi(t_{\l_{k-1}, k-1})\big)
\label{BP1}
\end{align}

\noi
for some $\l_k = \l_k(t) \in \{0, \dots, 2^kT\}$.
Then, 
from  \eqref{FL3}, \eqref{BP1}, 
and Minkowski's integral inequality
with \eqref{AP00}, 
we have 
\begin{align}
\begin{split}
\Big\| \| \Psi \|_{C[0, T]; \ft b^s_{p, \infty})}\Big\|_{L^r (\O)}
& \le \Big\|\| \Phi(t) \|_{C([0, T]; \F L^{s, p})}\Big\|_{L^r (\O)}\\
& \le  \sum_{k =-\log_2T}^\infty
\Big\|  \max_{0\leq \l_k \leq 2^kT} 
\| \Phi(t_{\l_k, k}) - \Phi(t_{\l'_{k-1}, k-1})\|_{\F L^{s, p}} \Big\|_{L^r (\O)}, 
\end{split}
\label{BP2}
\end{align}

\noi
where $t_{\l'_{k-1}, k-1}$ is one of the $2^{(k-1)}T$+1 equally spaced points such that 
\begin{align}
|t_{\l_k, k} - t_{\l'_{k-1}, k-1}| \leq 2^{-k}.
\label{BP3}
\end{align}

For $k \in  \Z\cap [-\log_2 T, \infty)$, let 
\begin{align*} 
q_k  = \max (\log 2^kT, p, r)
\sim \log (2^k T)+ r.
\end{align*}

\noi
Then, noting  that $(2^kT+1)^\frac{1}{q_k} \les 1$, 
it follows from 
\eqref{BP2} that 
\begin{align}
\begin{split}
& \Big\| \| \Phi \|_{C[0, T]; \ft b^s_{p, \infty})}\Big\|_{L^r(\O)}\\
& \quad \le  \sum_{k =-\log_2 T}^\infty
\bigg\| \bigg( \sum_{ \l_k = 0}^{2^kT} 
\| \Phi(t_{\l_k, k}) - \Phi(t_{\l'_{k-1}, k-1})\|_{\F L^{s, p}}^{q_k}\bigg)^\frac1{q_k} \bigg\|_{L^{q_k}(\O)} \\
& \quad =   \sum_{k =-\log_2 T}^\infty
 \bigg( \sum_{ \l_k = 0}^{2^kT} 
\Big\| \| \Phi(t_{\l_k, k}) - \Phi(t_{\l'_{k-1}, k-1})\|_{\F L^{s, p}}
\Big\|_{L^{q_k}(\O)}^{q_k} \bigg)^\frac1{q_k} \\
& \quad \les     \sum_{k =-\log_2 T}^\infty
 \max_{0\leq \l_k \leq 2^kT} 
\Big\| \| \Phi(t_{\l_k, k}) - \Phi(t_{\l'_{k-1}, k-1})\|_{\F L^{s, p}}
\Big\|_{L^{q_k}(\O)}.
\end{split}
\label{BP4}
\end{align}

\noi
From \eqref{FL2}, Minkowski's integral inequality,  and  \eqref{AP0}, 
we have 
\begin{align}
\begin{split}
\Big\| \|  & \Phi(t_{\l_k, k}) - \Phi(t_{\l'_{k-1}, k-1})\|_{\F L^{s, p}}
\Big\|_{L^{q_k}(\O)}\\
& = 
\Big\| \big\| \jb{n}^s\big( \ft \Phi(n, t_{\l_k, k}) - \ft \Phi(n, t_{\l'_{k-1}, k-1})\big) \big\|_{\l^p_n}
\Big\|_{L^{q_k}(\O)}\\
& \les \sqrt{q_k} 
\Big\|  \jb{n}^s \| \ft \Phi(n, t_{\l_k, k}) - \ft \Phi(n, t_{\l'_{k-1}, k-1})
\|_{L^{2}(\O)} \Big\|_{\l^p_n}\\
& =  \sqrt{q_k} 
\Bigg\|  \jb{n}^s
\bigg\|\int_{t_{\l'_{k-1}, k-1}}^{t_{\l_k, k}} e^{-it'n^3} d \be_n(t')\bigg\|_{L^2(\O)} 
\Bigg\|_{\l^p_n}\\
& \les \sqrt{\frac{q_k}{2^k}}, 
\end{split}
\label{BP5}
\end{align}

\noi
where the last step follows from \eqref{BP3}
and $sp < -1$.
Hence, from \eqref{BP4} and \eqref{BP5}, we obtain
\eqref{BP2} that 
\begin{align*}
\Big\| \| \Psi \|_{C[0, T]; \ft b^s_{p, \infty})}\Big\|_{L^r(\O)}
& \les \sqrt r   \sum_{k =-\log T}^\infty
\frac {\log 2^k + \log_2 T}{2^{\frac12k}} \\
& \les \sqrt r \sqrt {T\log T}.
\end{align*}

\noi
This proves \eqref{sto2}.
\end{proof}

\begin{remark}\label{REM:psi2}\rm

Let us consider the bound \eqref{sto1} 
when $I = [0, T]$, as discussed in Remark~\ref{REM:psi1}.
In this case, 
\eqref{AP1} becomes
\begin{align*}
\ft{\ind_{[0, T]} \Phi}(n, \tau) 
& = \int_0^{T} e^{-it'n^3} \int_{ t'}^{T}e^{-it\tau} dt d \be_n(t').
\end{align*}

\noi
In particular, the inner integral is estimated as
\begin{align}
\bigg|\int_{t'}^{T}e^{-it\tau}dt\bigg|
 \les \frac T{\jb\tau}.
\label{CP2}
\end{align}

\noi
Then, by repeating the computation above with \eqref{CP2}, 
we obtain \eqref{sto3}.

\end{remark}

\section{Pathwise bound on the iterated term with the stochastic convolution}
\label{SEC:B}

In this appendix, we establish a pathwise bound
on the $X^{-\al, 1- \al, T}$-norm
of $ \N_1(\Psi, u)$ appearing in \eqref{DD7}.
This was essentially carried out in \cite[``Estimate on (ii)'' on pp.\,296-297]{OH4}
but was done with an expectation.
In the following, based on the analysis in \cite{OH4}, 
we instead present straightforward pathwise analysis.
By duality, it suffices to estimate 
\begin{align}
\sum_{\substack{n, n_1 \in \Z\\n = n_1 + n_2}}
\intt_{\tau = \tau_1 + \tau_2} d\tau d\tau_1
\ind_{\s_1 = \MAX}
\frac{\jb{n}^{1-\al}d(n, \tau)}{\s_0^\al}
|\ft{\ind_{[0, T]}\Psi}(n_1, \tau_1)|
\frac{\jb{n_2}^{1-\al} |c(n_2, \tau_2)|}{\s^\al_2}, 
\label{BX1}
\end{align}

\noi
where
$\s_j$, $j = 0, 1, 2$, is as in \eqref{sig1}, 
$d = d(n, \tau)$ with $\|d\|_{\l^2_n L^2_\tau} = 1$, 
and 
$c(n, \tau) = \jb{n}^{-(1-\al)} \jb{\tau - n^3}^\al \ft u(n, \tau)$
such that $\|c\|_{\l^2_n L^2_\tau} = \|u\|_{X^{-(1-\al), \al}}$.

\medskip  

\noi
$\bullet$ {\bf Case 1:}
$\max(\s_0, \s_2) \ges \jb{n n_1n_2}^\frac{1}{100}$.
\\
\indent
Without loss of generality, assume 
$\s_0 \ges \jb{n n_1n_2}^\frac{1}{100}$.
Then, 
by \eqref{max} and the $L^4_{x, t}, L^2_{x, t}, L^4_{x, t}$-H\"older's inequality
followed by the $L^4$-Strichartz estimate \eqref{L4},  
we have
\begin{align}
\begin{split}
\eqref{BX1} 
& \les \sum_{\substack{n, n_1 \in \Z\\n = n_1 + n_2}}
\intt_{\tau = \tau_1 + \tau_2} d\tau d\tau_1
\frac{d(n, \tau)}{\s_0^{\al-200\dl}}
\jb{n_1}^{-\frac 12 - \dl} \s_1^{\frac 12 - \dl}|\ft{\ind_{[0, T]}\Psi}(n_1, \tau_1)|
\frac{ |c(n_2, \tau_2)|}{\s^\al_2}\\
& \les 
\|\ind_{[0, T]}\Psi\|_{X^{-\frac 12 - \dl, \frac 12 -\dl}}
\|u\|_{X^{-(1-\al), \al, T}}
\end{split}
\label{BX2}
\end{align}

\noi
by taking $\dl > 0$ sufficiently small.

\medskip  

\noi
$\bullet$ {\bf Case 2:}
$\max(\s_0, \s_2) \ll \jb{n n_1n_2}^\frac{1}{100}$.
\\
\indent
Define the set    $\O(n)$ by 
\begin{align*} 
\O(n) = \big\{   \s \in \R : &  \, \s = -3 n n_1 n_2 + o(\jb{n n_1 n_2}^\frac{1}{100})\\
& \,  \text{for some } n_1, n_2 \in \Z_*
\text{ with } n = n_1 + n_2 \big\}. 
\end{align*}

\noi
Then, we have
\begin{align}
\int \jb{\tau - n^3}^{-\frac 34} \ind_{\O(n)}(\tau - n^3) d\tau \les 1.
\label{OMG2}
\end{align}

\noi
See \cite[Lemma 5.3]{OH4}.
By \eqref{max}, the $L^4_{x, t}, L^2_{x, t}, L^4_{x, t}$-H\"older's inequality, 
the $L^4$-Strichartz estimate~\eqref{L4},  and 
H\"older's inequality (in $\tau$) with \eqref{OMG2}, 
we have
\begin{align}
\begin{split}
\eqref{BX1} 
& \les \sum_{\substack{n, n_1 \in \Z\\n = n_1 + n_2}}
\intt_{\tau = \tau_1 + \tau_2} d\tau d\tau_1
\frac{d(n, \tau)}{\s_0^{\al}}\\
& \hphantom{XXXXX} 
\times \ind_{\O(n_1)}(\tau_1 - n_1^3) \jb{n_1}^{-\frac 12 - \dl} \s_1^{\frac 12 + \dl}|\ft{\ind_{[0, T]}\Psi}(n_1, \tau_1)|
\frac{ |c(n_2, \tau_2)|}{\s^\al_2}\\
& \les \| \ind_{\O(n_1)}(\tau_1 - n_1^3)\jb{n_1}^{-\frac 12 - \dl} \s_1^{\frac 12 + \dl}\ft{\ind_{[0, T]}\Psi}(n_1, \tau_1)
\|_{\l^2_{n_1}L^2_{\tau_1}}\|u\|_{X^{-(1-\al), \al}}\\
& \les 
\|\ind_{[0, T]}\Psi\|_{Y^{-\frac 12 - \dl, \frac {11}{16} +\dl}_{2, 4}}\|u\|_{X^{-(1-\al), \al, T}}, 
\end{split}
\label{BX3}
\end{align}

\noi
where the $Y^{s, b}_{p, q}$-norm is defined in \eqref{Xsb3}.

\medskip
Given $N \in \NB$, a similar computation yields
\begin{align}
\begin{split}
& \| \N_1(\P_N^\perp \Psi, u)\|_{X^{-\al, 1- \al, T}}\\
& \quad \les 
\Big( \|\ind_{[0, T]}\P_N^\perp \Psi\|_{X^{-\frac 12 - \dl, \frac 12 -\dl}}
+ 
\|\ind_{[0, T]}\P_N^\perp \Psi\|_{Y^{-\frac 12 - \dl, \frac {11}{16} +\dl}_{2, 4}}
\Big)\|u\|_{X^{-(1-\al), \al, T}},
\end{split}
\label{BX4}
\end{align}

\noi
which motivates the definition of $\wt L_{\o, N}^\perp(T)$ in \eqref{F3}.

\begin{ack}\rm
T.O.~was supported by the European Research Council (grant no.~637995 ``ProbDynDispEq''
and grant no.~864138 ``SingStochDispDyn"). 
J.Q.~was
partially supported by an NSERC discovery grant.
P.S.~was partially supported by NSF grant DMS-1811093.
The authors would like to thank the anonymous referee for 
the helpful comments.
\end{ack}

\end{document}